\documentclass[]{siamart0216}

\usepackage{amsmath}
\usepackage{amssymb}
\usepackage{comment}
\usepackage{subfigure}
\usepackage{accents}
\usepackage{tcolorbox}
\usepackage{tikz,lipsum}
\usepackage{float}
\tcbuselibrary{skins,breakable}
\usetikzlibrary{calc,arrows}
\usepackage{mathtools}
\usepackage{bm}
\usepackage{url}
\usepackage{graphicx}

\newsiamthm{remark}{Remark}
\newsiamthm{assumption}{Assumption}

\newcommand{\real}{\ensuremath{\mathbb{R}}}

\newcommand{\subscr}[2]{#1_{\textup{#2}}}

\newcommand{\diag}[1]{\operatorname{diag}(#1)}

\newcommand{\eps}{\varepsilon}
\renewcommand{\epsilon}{\varepsilon}

\usepackage{psfrag}

\newcommand{\Rat}[1]{\ensuremath{\mathbb{R}^{#1}}}
\newcommand{\B}{\mathcal{B}}
\newcommand{\bone}{\mathbf{1}}

\newcommand{\F}{\mathcal{F}}
\renewcommand{\H}{\mathcal{H}}

\newcommand{\Obs}{\mathcal{O}}
\newcommand{\X}{\mathcal{X}}
\renewcommand{\P}{\mathcal{P}}
\newcommand{\bP}{\mathbb{P}}
\newcommand{\bN}{\mathbb{N}}
\newcommand{\fa}{\mathfrak a}
\newcommand{\fA}{\mathfrak A}
\newcommand{\fC}{\mathfrak C}
\newcommand{\fD}{\mathfrak D}
\newcommand{\fE}{\mathfrak E}
\newcommand{\fm}{\mathfrak m}
\newcommand{\fM}{\mathfrak M}
\newcommand{\fQ}{\mathfrak Q}
\newcommand{\fR}{\mathfrak R}

\newcommand{\ones}{\mathbf 1}

\newcommand{\bE}{\ensuremath{\mathbb{E}}}

\newcommand{\diameter}[2]{{\rm diam}_{#1}(#2)}
\newcommand{\supp}[1]{{\rm supp}(#1)}

\newcommand{\oprocendsymbol}{\hbox{$\square$}}
\newcommand{\oprocend}{\relax\ifmmode\else\unskip\hfill\fi\oprocendsymbol}

\newcommand{\longthmtitle}[1]{\mbox{}{\bf \textit{(#1).}}}
\allowdisplaybreaks

\makeatletter
\newcommand{\subalign}[1]{%
	\vcenter{%
		\Let@ \restore@math@cr \default@tag
		\baselineskip\fontdimen10 \scriptfont\tw@
		\advance\baselineskip\fontdimen12 \scriptfont\tw@
		\lineskip\thr@@\fontdimen8 \scriptfont\thr@@
		\lineskiplimit\lineskip
		\ialign{\hfil$\m@th\scriptstyle##$&$\m@th\scriptstyle{}##$\hfil\crcr
			#1\crcr
		}%
	}%
}
\makeatother

\usepackage{lipsum}
\usepackage{amsfonts}
\usepackage{graphicx}
\usepackage{epstopdf}
\usepackage{algorithmic}
\ifpdf
  \DeclareGraphicsExtensions{.eps,.pdf,.png,.jpg}
\else
  \DeclareGraphicsExtensions{.eps}
\fi

\newcommand{\TheTitle}{High-Confidence Data-Driven Ambiguity Sets for
  Time-Varying Linear Systems}

\newcommand{\TheAuthors}{D. Boskos, J. Cort{\'e}s, and S. Mart{\'i}nez}

\headers{Data-Driven Ambiguity Sets for Linear Systems}{\TheAuthors}

\title{{\TheTitle}\thanks{This work was supported by the DARPA
    Lagrange program through award N66001-18-2-4027. A preliminary
    version of this paper appeared as~\cite{DB-JC-SM:20-acc} at the
    American Control Conference.}}

\author{Dimitris Boskos\thanks{Delft Center for Systems and Control, 
    Delft University of Technology (\email{d.boskos@tudelft.nl}).}
  \and
 Jorge Cort{\'e}s\thanks{Department of
 	Mechanical and Aerospace Engineering, University of California,
 	San Diego (\email{cortes,soniamd@ucsd.edu}).}
 \and
 Sonia Mart{\'i}nez\footnotemark[3]
 }

\usepackage{amsopn}

\ifpdf
\hypersetup{
  pdftitle={\TheTitle},
  pdfauthor={\TheAuthors}
}
\fi

\begin{document}

\maketitle

\begin{abstract}
	This paper builds Wasserstein ambiguity sets for the unknown
	probability distribution of dynamic random variables leveraging
	noisy partial-state observations. The constructed ambiguity sets
	contain the true distribution of the data with quantifiable
	probability and can be exploited to formulate robust stochastic
	optimization problems with out-of-sample guarantees. We assume the
	random variable evolves in discrete time under uncertain initial
	conditions and dynamics, and that noisy partial measurements are
	available. All random elements have unknown probability
	distributions and we make inferences about the distribution of the
	state vector using several output samples from multiple realizations
	of the process. To this end, we leverage an observer to estimate the
	state of each independent realization and exploit the outcome to
	construct the ambiguity sets.  We illustrate our results in an
	economic dispatch problem involving distributed energy resources
	over which the scheduler has no direct control.
\end{abstract}








\section{Introduction}


Decisions under uncertainty are ubiquitous in a wide range of
engineering applications.  Faced with complex systems that include
components with probabilistic models, such decisions seek to provide
rigorous solutions with quantifiable guarantees in hedging against
uncertainty.  In practice, the designer makes inferences about
uncertain elements based on collected data and exploits them to
formulate data-driven stochastic optimization problems. This
decision-making paradigm has found applications in finance,
communications, control, medicine, and machine learning.  Recent
research focuses on how to retain high-confidence guarantees for the
optimization problems under plausible variations of the data.  To this
end, distributionally robust optimization (DRO) formulations evaluate
the optimal worst-case performance over an ambiguity set of
probability distributions that contains the true one with high
confidence. Such ambiguity sets are typically constructed under the
assumption that data are generated from a static distribution and can
be measured in a direct manner.  In this paper we significantly expand
on the class of scenarios for which reliable ambiguity sets can be
constructed.  We consider scenarios where the random variable is
dynamic and partial measurements, corrupted by noise, are
progressively collected from its evolving distribution.  In our
analysis, we exploit the underlying dynamics and study how the
probabilistic properties of the noise affect the ambiguity set size
while maintaining the same guarantees.

\textit{Literature review:} Optimal decision problems in the face of
uncertainty, like expected-cost minimization and chance-constrained
optimization, are the cornerstones of stochastic
programming~\cite{AS-DD-AR:14}.
Distributionally robust versions of stochastic optimization
problems~\cite{AB-LEG-AN:09,JB-KM:19,AS:17} carry out a worst-case
optimization over all possibilities from an ambiguity set of
probability distributions. This is of particular importance in
data-driven scenarios where the unknown distributions of the random
variables are inferred in an approximate manner using a finite amount
of data~\cite{DB-VG-NK:18a}.  To hedge this uncertainty, optimal
transport ambiguity sets have emerged as a promising tool. These sets
typically group all distributions up to some distance from the
empirical approximation in the Wasserstein metric~\cite{CV:03}.  There
are several reasons that make this metric a popular choice among the
distances between probability distributions, particularly, for
data-driven problems. Most notably, the Wasserstein metric penalizes
horizontal dislocations between distributions and provides ambiguity
sets that have finite-sample guarantees of containing the true
distribution and lead to tractable optimization problems. This has
rendered the convergence of empirical measures in the Wasserstein
distance an ongoing active research
area~\cite{JD-FM:19,SD-MS-RS:13,NF-AG:15,BK18,JW-FB:19,JW-QB:19}.
Towards the exploitation of Wasserstein ambiguity sets for DRO
problems, the work~\cite{PME-DK:17} introduces tractable
reformulations with finite-sample guarantees, further exploited
in~\cite{ZC-DK-WW:18,AH-AC-JL:19} to deal with distributionally robust
chance-constrained programs.
The work~\cite{AC-JC:20-tac} develops distributed optimization
algorithms using Wasserstein balls, while optimal transport ambiguity
sets have recently been connected to regularization for machine
learning~\cite{JB-YK-KM:16,RG-XC-AJK:17,SS-DK-PME:19}. The
paper~\cite{DL-SM:18-cdc} exploits Wasserstein balls to robustify
data-driven online optimization algorithms,
and~\cite{SSA-VAN-DK-PME:18} leverages them for the design of
distributionally robust Kalman filters.
Further applications of Wasserstein ambiguity sets include the
synthesis of robust control policies for Markov decision
processes~\cite{IY:17} and their data-driven extensions~\cite{IY:18b},
and regularization for stochastic predictive control
algorithms~\cite{JC-JL-FD:19}.
Several recent works have also devoted attention to distributionally 
robust problems in power systems control, including optimal power flow 
\cite{YG-KB-ED-ZH-THS:18,BL-JM-RJ:18} and economic dispatch 
\cite{FX-BMH-LF-DE-KC:19,JL-YC-CD-JL-JL:20,BKP-ARH-SB-DSC-AC:20}.
Time-varying aspects of Wasserstein ambiguity sets are considered in
our previous work: in~\cite{DL-DF-SM:19-ecc} for dynamic traffic
models, in~\cite{DL-DF-SM:20-lcss-arxiv} for online learning of
unknown dynamical environments, in~\cite{DB-JC-SM:21-tac}, which
constructs ambiguity balls using progressively assimilated dynamic
data for processes with random initial conditions that evolve under
deterministic dynamics, and in~\cite{FB-DB-JC-SM-DMT:21}, which
studies the propagation of ambiguity bands under hyperbolic PDE
dynamics.  In contrast, in the present work, the state distribution
does not evolve deterministically due to the presence of random
disturbances, which together with output measurements that are
corrupted by noise, generate additional stochastic elements that make
challenging the quantification of the ambiguity set guarantees.



%
\textit{Statement of contributions:} Our contributions revolve around
building Wasserstein ambiguity sets with probabilistic guarantees for
dynamic random variables when we have no knowledge of the probability
distributions of their initial condition, the disturbances in their
dynamics, and the measurement noise.
To this end, our first contribution estimates the states of several
process realizations from output samples and exploits these 
estimates to build a suitable empirical distribution as the center 
of an ambiguity ball.
Our second contribution is the exploitation of concentration of
measure results to quantify the radius of this ambiguity ball so that
it provably contains the true state distribution with high
probability.  To achieve this, we break the radius into nominal and
noise components. The nominal component captures the deviation between
the true distribution and the empirical distribution formed by the
state realizations.  The noise component captures the deviation
between the empirical distribution and the center of our ambiguity
ball. To quantify the latter, we carefully evaluate the impact of the
estimation error, which due to the measurement noise, does not have a
compactly supported distribution like the internal uncertainty and
requires a separate analysis.  Our third contribution is the extension
of these results to obtain \textit{simultaneous} guarantees about
ambiguity sets that are built along finite time horizons, instead of
at isolated time instances.  The fourth contribution is to generalize
a concentration inequality around the mean of sufficiently
light-tailed independent random variables, which enables us to obtain
tighter results when analyzing the effect of the estimation error.
%
%
%
Our last contribution is the validation of the results in simulation
for a distributionally robust economic dispatch problem, for which we
further provide a tractable reformulation.  We stress that our general
objective revolves around the robust uncertainty quantification (i.e.,
distributional inference) problem at hand, without having DRO as a
necessary end-goal. Further, our approach is fundamentally different
from classical Kalman filtering, where the initial state and dynamics
noise distributions are known and Gaussian, and hence, the state
distribution over time is also a known Gaussian random variable. Here,
instead, we are interested to infer the unknown state distribution
from data collected by multiple realizations of the dynamics. For each
such realization, we use an observer
%
%
since we have no concrete knowledge of the state and noise random 
models to directly invoke optimal filtering techniques.

\section{Preliminaries}\label{sec:prelims}

Here we present general notation and concepts from probability theory
used throughout the paper.

\textit{Notation:} We denote by $\|\cdot\|_p$ the $p$th norm in
$\Rat{n}$, $p\in[1,\infty]$, using also the notation $\|\cdot\|\equiv
\|\cdot\|_2$ for the Euclidean norm. The inner product of two vectors
$a,b\in\Rat{n}$ is denoted by $\langle a,b\rangle$ and the Khatri-Rao
product~\cite{SL:99} of $\bm a \equiv(a^1,\ldots,a^d)\in\Rat{d}$ and
$\bm b\equiv(b^1,\ldots,b^d)\in\Rat{dn}$, with each $b^i $ belonging
to $\Rat{n}$, is $\bm a*\bm b:=(a^1b^1,\ldots,a^db^d)\in\Rat{dn}$.  We
use the notation $B_p^n(\rho)$ for the ball of center zero and radius
$\rho$ in $\Rat{n}$ with the $p$th norm and $[n_1:n_2]$ for the set of
integers $\{n_1,n_1+1,\ldots,n_2\}\subset\bN\cup\{0\} =:\bN_0$. The
interpretation of a vector in $\Rat{n}$ as an $n\times 1$ matrix
should be clear form the context (this avoids writing double
transposes).  The diameter of a set $S\subset\Rat{n}$ with the $p$th
norm is defined as $\diameter{p}{S}:=\sup\{\|x-y\|_p\,|\,x,y\in S\}$
and for $z\in\Rat{n}$, $S+z:=\{x+z\,|\,x\in S\}$. We denote the
induced Euclidean norm of a matrix $A\in\Rat{m\times n}$ by
$\|A\|:=\max_{\|x\|=1}\|Ax\|/\|x\|$. Given $B\subset\Omega$, $\bone_B$
is the indicator function of $B$ on $\Omega$, with $\bone_B(x)=1$ for
$x\in B$ and $\bone_B(x)=0$ for $x\notin B$.

\textit{Probability Theory:} We denote by $\B(\Rat{d})$ the Borel
$\sigma$-algebra on $\Rat{d}$, and by $\P(\Rat{d})$ the probability
measures on $(\Rat{d},\B(\Rat{d}))$.  For any $p\ge 1$,
$\P_p(\Rat{d}):=\{\mu\in\P(\Rat{d})\,|\,
\int_{\Rat{d}}\|x\|^pd\mu<\infty\}$ is the set of probability measures
in $\P(\Rat{d})$ with finite $p$th moment. The Wasserstein distance
between $\mu,\nu\in\P_p(\real^d)$ is
\begin{align*}
	W_p(\mu,\nu):=\Big(\inf_{\pi\in\H(\mu,\nu)} \Big\{
	\int_{\Rat{d}\times\Rat{d}}\|x-y\|^p \pi(dx,dy)\Big\}\Big)^{1/p},
\end{align*}
where $\H(\mu,\nu)$ is the set of all couplings between $\mu$ and
$\nu$, i.e., probability measures on $\Rat{d}\times\Rat{d}$ with
marginals $\mu$ and $\nu$, respectively. For any $\mu\in\P(\Rat{d})$,
its support is the closed set
$\supp{\mu} := \{x\in\Rat{d}\,|\,\mu(U)>0\;\textup{for each
	neighborhood}\;U\;{\rm of}\;x\}$,
%
%
%
or equivalently, the smallest closed set with measure one. For a
random variable $X$ with distribution $\mu$ we also denote
$\supp{X}\equiv\supp{\mu}$. We denote the product of the distributions
$\mu$ in $\Rat{d}$ and $\nu$ in $\Rat{r}$ by the distribution
$\mu\otimes\nu$ in $\Rat{d}\times\Rat{r}$. The convolution
$\mu\star\nu$ of the distributions $\mu$ and $\nu$ on $\Rat{d}$ is the
image of the measure $\mu\otimes\nu$ on $\Rat{d}\times\Rat{d}$ under
the mapping $(x,y)\mapsto x+y$; equivalently,
$\mu\star\nu(B)=\int_{\Rat{d}\times\Rat{d}}\bone_B(x+y)\mu(dx)\nu(dy)$
for any $B\in\B(\Rat{d})$ (c.f. \cite[Pages 207,
208]{VB:07-vol1}). Given a measurable space $(\Omega,\F)$, an exponent
$p\ge 1$, the convex function $\real \ni x \mapsto
\psi_p(x):=e^{x^p}-1$, and the linear space of scalar random variables
$L_{\psi_p}:=\{X\,|\,\bE[\psi_p(|X|/t)]<\infty\;\textup{for
	some}\;t>0\}$ on $(\Omega,\F)$, the $\psi_p$-Orlicz norm
(cf.~\cite[Section 2.7.1]{RV:18}) of $X\in L_{\psi_p}$ is 
\begin{align*}
	\|X\|_{\psi_p}:=\inf\{t>0\,|\,\bE[\psi_p(|X|/t)]\le 1\}.
\end{align*}  
When $p=1$ and $p=2$, each random variable in $L_{\psi_p}$ is
sub-exponential and sub-Gaussian, respectively. We also denote by
$\|X\|_p\equiv\big(\bE\big[|X|^p\big]\big)^{\frac{1}{p}}$ the norm of
a scalar random variable with finite $p$th moment, i.e., the classical
norm in $L^p(\Omega)\equiv L^p(\Omega;P_X)$, where $P_X$ is the
distribution of $X$. The interpretation of $\|\cdot\|_p$ as the $p$th
norm of a vector in $\Rat{n}$ or a random variable in $L^p$ should be
clear from the context throughout the paper.
Given a set $\{X_i\}_{i\in I}$ of random variables, we denote by
$\sigma(\{X_i\}_{i\in I})$ the $\sigma$-algebra generated by them.  We
conclude with a useful technical result which follows from Fubini's 
theorem~\cite[Theorem 2.6.5]{RBA:72}.

\begin{lemma}\longthmtitle{Expectation
		inequality}\label{lemma:expectation:inequality}
	Consider the independent random vectors $X$ and $Y$, taking values
	in $\Rat{n_1}$ and $\Rat{n_2}$, respectively, and let $(x,y) \mapsto
	g(x,y)$ be integrable.  Assume that $\bE[g(x,Y)]\ge k(x)$ for some
	integrable function $k$ and all $x\in
	K$ with $\supp{X}\subset K\subset\Rat{n_1}$. Then, 
	$\bE[g(X,Y)]\ge\bE[k(X)]$.
\end{lemma}

\section{Problem formulation} \label{sec:pf}

Consider a stochastic optimization problem where the objective
function $x\mapsto f(x,\xi)$ depends on a random variable~$\xi$ with
an \textit{unknown} distribution $P_\xi$. To hedge this uncertainty,
rather than using the empirical distribution
\begin{align} \label{empirical:distribution}
	P_\xi^N:=\frac{1}{N}\sum_{i=1}^N\delta_{\xi^i},
\end{align}
formed by $N$ i.i.d.~samples $\xi^1,\ldots,\xi^N$ of $P_\xi$ to
optimize a sample average approximation of the expected value of~$f$,
one can instead consider the 
DRO problem
\begin{align}\label{eq:DRO}
	\inf_{x\in\X}\sup_{P\in\P^N}\bE_P[f(x,\xi)],
\end{align}
of evaluating the worst-case expectation over some \textit{ambiguity
	set} $\P^N$ of probability measures. This helps the designer
robustify the decision against plausible variations of the data, which
can play a significant role when the number of samples is limited.
Different approaches exist to construct the ambiguity set $\P^N$ so
that it contains the true distribution $P_{\xi}$ with high
confidence. We are interested in approaches that employ data, and in
particular the empirical distribution $P_\xi^N$, to construct them. In
the present setup, the data is generated by a dynamical system subject
to disturbances, and we only collect partial (instead of full)
measurements that are distorted by noise. Therefore, it is no longer
obvious how to build a candidate state distribution as
in~\eqref{empirical:distribution} from the collected samples. Further,
we seek to address this in a distributionally robust way, i.e.,
finding a suitable replacement $\widehat P_\xi^N$
for~\eqref{empirical:distribution} together with an associated
ambiguity set, by exploiting the dynamics of the underlying process.

To make things precise, consider data generated by a discrete-time
system
\begin{subequations}\label{eq:data}
	\begin{align}\label{data:state}
		\xi_{k+1}= A_k\xi_k+G_kw_k,\quad \xi_k\in\Rat{d}, \quad
		w_k\in\Rat{q},
	\end{align}
	with linear output
	\begin{align} \label{data:output} \zeta_k = H_k\xi_k+v_k,\quad
		\zeta_k\in\Rat{r}.
	\end{align}
\end{subequations}
The initial condition $\xi_0$ and the noises $w_k$ and $v_k$,
$k\in\mathbb N_0$ in the dynamics and the measurements, respectively,
are random variables with an \textit{unknown} distribution. We seek to
build an ambiguity set for the state distribution at certain time
$\ell\in\mathbb N$, by collecting data up to time $\ell$ from multiple
independent realizations of the process, denoted by $\xi^i$,
$i\in[1:N]$.  This can occur, for instance, when the same process is
executed repeatedly, or in multi-agent scenarios where identical
entities are subject to the same dynamics, see e.g.~\cite{SZ:19}.
%
The time-dependent matrices in the dynamics~\eqref{eq:data} widen the 
applicability of the results, since they can capture the linearization of 
nonlinear systems along trajectories or the sampled-data analogues of
continuous-time systems under irregular sampling, even if the latter
are linear and time invariant.  To formally describe the problem, we
consider a probability space $(\Omega,\F,\bP)$ containing all random
elements from these realizations, and make the following sampling
assumption.

\begin{assumption}\longthmtitle{Sampling
		schedule}\label{samplig:assumption}
	For each realization $i$ of system~\eqref{eq:data}, output samples
	$\zeta_0^i,\ldots,\zeta_{\ell}^i$ are collected over the discrete
	time instants of the sampling horizon $[0:\ell]$.
\end{assumption}
%
%
According to this assumption, the measurements of all realizations are
collected over the same time window $[0:\ell]$. 
%
To obtain quantifiable characterizations of the ambiguity sets, we
require some further hypotheses on the classes of the distributions
$P_{\xi_0}$ of the initial condition, $P_{w_k}$ of the dynamics noise,
and $P_{v_k}$ of the measurement errors
(cf.~Figure~\ref{fig:distribution:classes}).
These assumptions are made for individual realizations and allow us to 
consider 
non-identical observation error distributions---in this way, we allow for 
the 
case where each realization is measured by a non-identical sensor of 
variable 
precision.

\begin{assumption}\longthmtitle{Distribution
		classes}\label{assumption:distributiuon:class} Consider a finite
	sequence of 
	realizations $\xi^i$, $i\in[1:N]$ of \eqref{data:state} with
	associated outputs given by \eqref{data:output}, and noise elements
	$w_k^i$, $v_k^i$, $k\in\bN_0$. We assume the following:
	
	\noindent \textbf{H1:} The distributions $P_{\xi_0^i}$, $i\in[1:N]$,
	are identically distributed; further $P_{w_k^i}$, $i\in[1:N]$, are
	identically distributed for all $k\in\bN_0$.
	
	\noindent \textbf{H2:} The sigma fields
	$\sigma(\{\xi_0^i\}\cup\{w_k^i\}_{k\in\bN_0}\big)$,
	$\sigma\big(\{v_k^i\}_{k\in\bN_0}\big)$, $i\in[1:N]$ are independent.
	
	\noindent \textbf{H3:} The supports of the distributions $P_{\xi_0^i}$ and 
	$P_{w_k^i}$, $k\in\bN_0$ are compact, centered at the origin, and have 
	diameters $2\rho_{\xi_0}$ and $2\rho_w$, respectively, for all $i$.
	
	\noindent \textbf{H4:} The components of the random vectors $v_k^i$ have 
	uniformly bounded $L^p$ and $\psi_p$-Orlicz norms, as follows,  
	\begin{align*}
		0<m_v\le \|v_{k,l}^i\|_p\le M_v,\quad \|v_{k,l}^i\|_{\psi_p}\le C_v, 
	\end{align*}
	for all $k\in\bN_0$, $i\in[1:N]$, and $l\in[1:r]$, where $p\ge 1$. 
\end{assumption}

\begin{remark}\longthmtitle{Bounded $\psi_p$-Orlicz/$L_p$-norm ratio} 
	{\rm By definition, $\psi_p$-Orlicz norms can become significantly
		larger than $L_p$ norms for random variables with heavier
		tails. Thus, over an infinite sequence of random variables
		$\{X_k\}$, the ratio $\|X_k\|_{\psi_p}/\|X_k\|_p$ may grow
		unbounded. We exclude this by assuming that $C_v$ and
		$m_v$ are either positive or zero simultaneously, in which case we
		set $C_v/m_v:=0$.}  \oprocend
\end{remark}

\begin{figure}[tbh]
	\centering
	\includegraphics[width=.55\columnwidth]{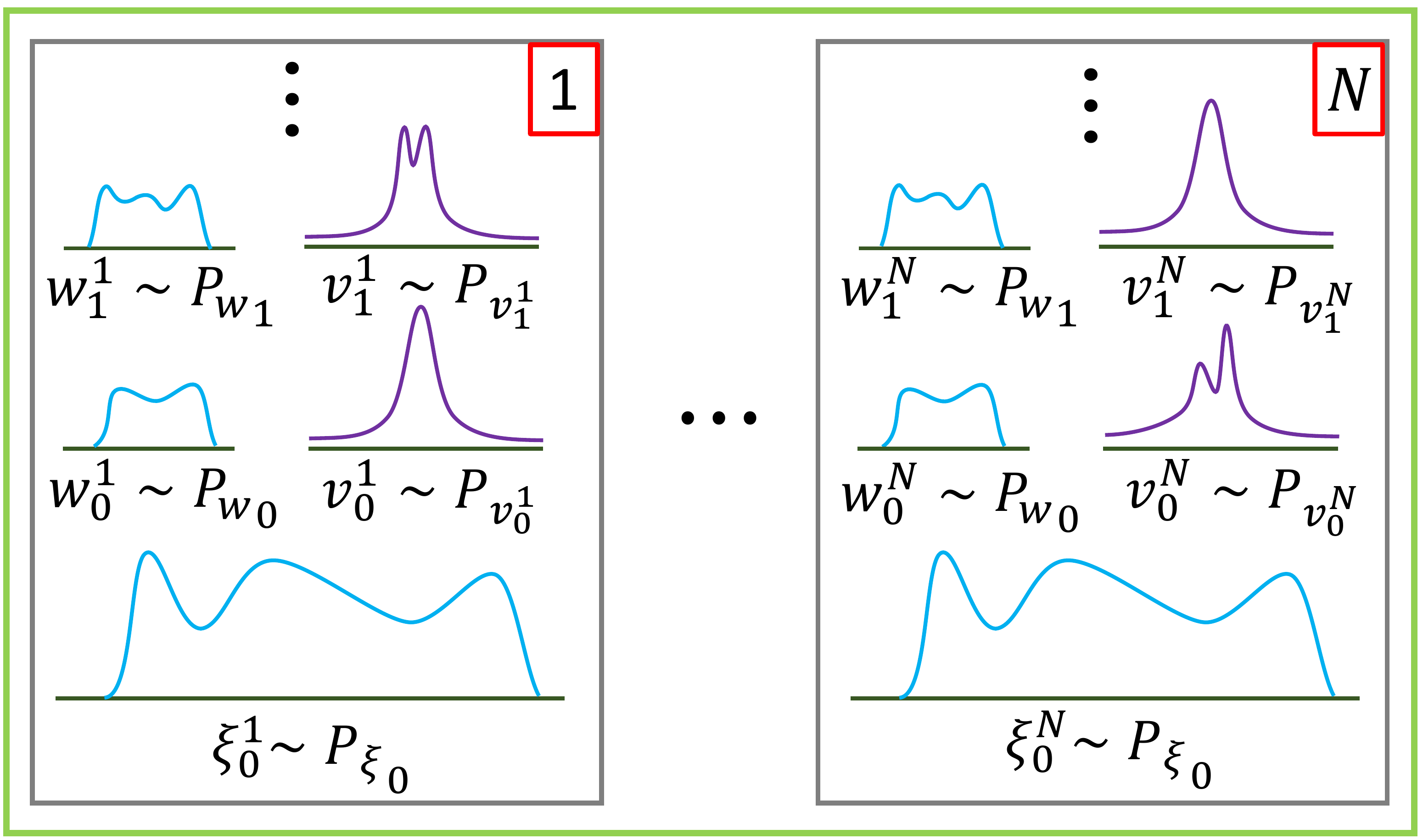}
	\caption{Illustration of the probabilistic models for the random
		variables in the dynamics and observations according to
		Assumption~\ref{assumption:distributiuon:class}.}
	\label{fig:distribution:classes}
\end{figure}

%
A direct approach to build the ambiguity set using the measurements of
the trajectories at time $\ell$ would be severely limited, since the
output map is in general not invertible. In such case, the inverse
image of each measurement is the translation of a subspace, whose
location is further obscured by the measurement noise. As a
consequence, candidate states for a generated output sample may lie at
an arbitrary distance apart, which could only be bounded by making
additional hypotheses about the support of the state
distribution. Instead, despite the lack of full-state information, we
aim to leverage the system dynamics to estimate the state from the
whole assimilated output trajectory.
%
To guarantee some boundedness notion for the state estimation
errors over arbitrary evolution horizons, we make the following
assumption.

\begin{assumption}\label{detectability:assumption}
	\longthmtitle{Detectability/uniform observability}
	System~\eqref{eq:data} satisfies one of the following properties:
	
	\noindent (i) It is time invariant and the pair $(A,H)$ (with $A\equiv 
	A_k$ and $H\equiv H_k$) is detectable. 
	
	\noindent (ii) It is uniformly observable, i.e., for some $t\in\mathbb N$, 
	the observability Gramian  
	\begin{align*}
		\Obs_{k+t,k}:=\sum_{i=k}^{k+t}\Phi_{i,k}^\top H_i^\top H_i\Phi_{i,k} 
	\end{align*}
	satisfies $\Obs_{k+t,k}\succeq bI$ for certain $b>0$ and all
	$k\in\bN_0$, where  we denote $\Phi_{k+s,k}:=A_{k+s-1}\cdots$ 
	\newline $A_{k+1}A_k$.
	Further, all system matrices are uniformly bounded and the singular
	values of $A_k$ and the norms of $\|H_k\|$ are uniformly bounded
	below.
\end{assumption}

\textit{Problem statement:} Under Assumptions~\ref{samplig:assumption}
and~\ref{assumption:distributiuon:class} on the measurements and
distributions of $N$ realizations of the system~\eqref{eq:data}, we
seek to construct an estimator $\widehat
\xi_\ell^i(\zeta_0^i,\ldots,\zeta_\ell^i)$ for the state of each
realization and build an ambiguity set for the state distribution at
time $\ell$ with probabilistic guarantees. Further, under
Assumption~\ref{detectability:assumption} on the system's
detectability/uniform observability properties, we aim to characterize
the effect of the estimation precision on the accuracy of the
ambiguity sets.

We proceed to address the problem in
Section~\ref{sec:state:estimator:ambig} by exploiting a Luenberger
observer to estimate the states of the collected data and using them
to replace the classical empirical
distribution~\eqref{empirical:distribution} in the construction of the
ambiguity set.  To obtain the probabilistic guarantees, we leverage
concentration inequalities to bound the distance between the updated
empirical distribution and the true state distribution with high
confidence.  To this end, we further quantify the increase of the
ambiguity radius due to the noise.  We also study the beneficial
effect on the ambiguity radius of detectability/uniform observability
for arbitrarily long evolution horizons in
Section~\ref{sec:uniform:noise:radius:bnd}.

\section{State-estimator based ambiguity sets} 
\label{sec:state:estimator:ambig}

We address here the question of how to construct an ambiguity set at
certain time instant $\ell$, when samples are collected from
\eqref{eq:data} according to Assumption~\ref{samplig:assumption}. If
we had access to $N$ independent full-state samples
$\xi_\ell^1,\ldots,\xi_\ell^N$ from the distribution of $\xi$ at
$\ell$, we could construct an ambiguity ball in the Wasserstein metric
$W_p$ centered at the empirical
distribution~\eqref{empirical:distribution} with
$\xi^i\equiv\xi_\ell^i$ and containing the true distribution with high
confidence. In particular, for any confidence $1-\beta>0$, it is
possible, cf. \cite[Theorem 3.5]{PME-DK:17}, to specify an ambiguity
ball radius $\eps_N(\beta)$ so that the true distribution of
$\xi_\ell$ is in this ball with confidence $1-\beta$, i.e.,
\begin{align*}
	\bP(W_p(P_{\xi_\ell}^N,P_{\xi_\ell})\le\eps_N(\beta))\ge 1-\beta .
\end{align*}
Instead, since we only can collect noisy partial measurements of the
state, we use a Luenberger observer to estimate $\xi$ at time~$\ell$.
The dynamics of the observer, initialized at zero, is given by
\begin{align}\label{observer:dynamics}
	\widehat{\xi}_{k+1} =
	A_k\widehat{\xi}_k+K_k(H_k\widehat{\xi}_k-\zeta_k), \qquad
	\widehat{\xi}_0=0,
\end{align}
where each $K_k$ is a nonzero gain matrix.  Using the corresponding
estimates from system \eqref{observer:dynamics} for the independent
realizations of \eqref{data:state}, we define the \textit{(dynamic)
	estimator-based empirical distribution}
\begin{align} \label{dyn:estim:empirical:distribution} \widehat
	P_{\xi_k}^N:=\frac{1}{N}\sum_{i=1}^N\delta_{\widehat{\xi}_k^i},
\end{align}
Denoting by $e_k:=\xi_k-\widehat{\xi}_k$ the error
between~\eqref{data:state} and the observer \eqref{observer:dynamics},
the error dynamics is $ e_{k+1} = F_ke_k+G_kw_k+K_kv_k$, $e_0=\xi_0$,
where $F_k:=A_k+K_kH_k$ and $\xi_0$ is the initial condition
of~\eqref{data:state}. In particular,
\begin{align} \label{error:formula} e_k =
	\Psi_k\xi_0+\sum_{\kappa=1}^k\big(\Psi_{k,k-\kappa+1}
	G_{k-\kappa}w_{k-\kappa}+\Psi_{k,k-\kappa+1}K_{k-\kappa}v_{k-\kappa}\big)
\end{align}
for all $k\ge 1$, where $\Psi_{k+s,k}:=F_{k+s-1}\cdots F_{k+1}F_k$,
$\Psi_{k,k}:=I$ and $\Psi_k:=\Psi_{k,0}$. To build the ambiguity set
at time $\ell$, we set its center at the estimator-based empirical
distribution $\widehat P_{\xi_\ell}^N$ given
by~\eqref{dyn:estim:empirical:distribution}. In what follows, we
leverage concentration of measure results to identify an ambiguity
radius $\psi_N(\beta)$ so that the resulting Wasserstein ball contains
the true distribution with a given confidence $1-\beta$.  Note that
even if a distributionally robust framework is not employed, replacing
the empirical distribution by the estimator empirical distribution in
\eqref{dyn:estim:empirical:distribution} does no longer guarantee
consistency, in the sense that the estimator empirical distribution
\textit{does not necessarily converge (weakly) to the true
	distribution}. Hence, there is also no indication that the solution
to the associated estimator Sample Average Approximation (SAA)
problem, i.e., to
\begin{align*}
	\inf_{x\in\X}\frac{1}{N}\sum_{i=1}^N f(x,\xi_\ell^i)
\end{align*}
with $\xi_\ell^i$ replaced by $\widehat\xi_\ell^i$, will be a
consistent estimator of the solution to the nominal stochastic
optimization problem.  This is a fundamental limitation that is
justified by the fact that, in general, the estimation error is
dependent on the state realization, i.e., it has a variable
distribution when conditioned on the state and the internal noise,
and so its effect cannot be easily reversed
(this may only be possible in rather degenerate cases, e.g., one has
access to full-sate samples and the measurement noise is known).
%

Note that the random variable $\xi_k^i$ of a system realization at
time $k$ is a function $\xi_k^i(\xi_0^i,\bm w_k^i)$ of the random
initial condition $\xi_0^i$ and the dynamics noise $\bm
w_k^i\equiv(w_0^i,\ldots,w_{k-1}^i)$.  Analogously, the estimated
state $\widehat{\xi}_k^i$ of each observer realization is a stochastic
variable $\widehat{\xi}_k^i(\xi_0^i,\bm w_k^i,\bm v_k^i)$ with
additional randomness induced by the output noise $\bm
v_k^i\equiv(v_0^i,\ldots,v_{k-1}^i)$. Using the compact notation
$\bm{\xi}_0\equiv(\xi_0^1,\ldots,\xi_0^N)$, $\bm w_k\equiv(\bm
w_k^1,\ldots,\bm w_k^N)$, and $\bm v_k\equiv(\bm v_k^1,\ldots,\bm
v_k^N)$ for the corresponding initial conditions, dynamics noise, and
output noise of all realizations, respectively, we can denote the
empirical and estimator-based-empirical distributions at time $\ell$ 
as $P_{\xi_\ell}^N(\bm{\xi}_0,\bm w_\ell)$ and $\widehat
P_{\xi_\ell}^N(\bm{\xi}_0,\bm w_\ell,\bm v_\ell)$.  If we view the
initial conditions and the corresponding internal noise of the
realizations $\xi^i$ over the whole time horizon as deterministic
quantities, we use the alternative notation $P_{\xi_\ell}^N(\bm
z,\bm{\omega})$ and $\widehat P_{\xi_\ell}^N(\bm z,\bm{\omega},\bm
v_\ell)$ for the corresponding distributions, where $\bm
z=(z^1,\ldots,z^N)$, $z^1\equiv\xi_0^1,\ldots,z^N\equiv\xi_0^N$, and
$\bm{\omega}=(\bm{\omega}^1,\ldots,\bm{\omega}^N)$,
$\bm{\omega}^1\equiv\bm w_{\ell}^1,\ldots,\bm{\omega}^N\equiv\bm
w_{\ell}^N$.  We also denote by $P_{\xi_{\ell}}$ the true distribution
of the data at discrete time $\ell$, where from~\eqref{data:state},
\begin{align} \label{pushed:forward:ditributiuon} 
	\xi_{\ell}=\Phi_\ell\xi_0+\sum_{k=1}^\ell\Phi_{\ell,\ell-k+1}
	G_{\ell-k}w_{\ell-k},
\end{align}
where $\Phi_\ell:=\Phi_{\ell,0}$ and $\Phi_{\ell,\ell}:=I$ (and with
$\Phi_{k+\delta k,k}$ defined in
Assumption~\ref{detectability:assumption}).  Then, it follows from
\textbf{H1} and \textbf{H2} in
Assumption~\ref{assumption:distributiuon:class} that the random states
$\xi_{\ell}^i$ of the system realizations are independent and
identically distributed. Leveraging this, our goal is to associate to
each confidence $1-\beta$, an ambiguity radius $\psi_N(\beta)$ so that
\begin{align}\label{ambiguity:set:confidence}
	\bP(W_p(\widehat P_{\xi_{\ell}}^N,P_{\xi_{\ell}})\le\psi_N(\beta))\ge 
	1-\beta.
\end{align}
To achieve this, we decompose the confidence as the product of two
factors:
\begin{align} \label{confidence:decomposition}
	1-\beta=(1-\subscr{\beta}{nom})(1-\subscr{\beta}{ns}).
\end{align}
%
%
%
The first factor (the nominal component ``nom'') is exploited to
control the Wasserstein distance between the empirical
distribution and the true state distribution~$P_{\xi_\ell}$.  The
purpose of the second factor (the noise component ``ns'') is to bound
the Wasserstein distance between the true- and the
estimator-based-empirical distributions, that 
is affected by the measurement noise.
Using this decomposition, our strategy to get
\eqref{ambiguity:set:confidence} builds on further breaking the
ambiguity radius as
\begin{align}\label{psi:dfn}
	\psi_N(\beta):= \eps_N(\subscr{\beta}{nom})
	+\widehat\eps_N(\subscr{\beta}{ns}) .
\end{align}
%
%
We exploit what is known~\cite{DB-JC-SM:21-tac} for the no-noise case
to bound the \textit{nominal ambiguity radius}
$\eps_N(\subscr{\beta}{nom})$ with confidence $1-\subscr{\beta}{nom}$.
Moreover, we bound the \textit{noise ambiguity radius}
$\widehat\eps_N(\subscr{\beta}{ns})$ with confidence
$1-\subscr{\beta}{ns}$.  This latter radius corresponds to the impact
on distributional uncertainty of the internal and measurement
noise. In the next sections we present the precise individual bounds
for these terms and then combine them to obtain the overall ambiguity
radius.

\subsection{Nominal ambiguity radius} \label{subsec:nominal:radius}

According to Assumption~\ref{assumption:distributiuon:class}, the
initial condition and internal noise distributions are compactly
supported, and hence, the same holds also for the state distribution
along time. We will therefore use the following result, that is
focused on compactly supported distributions and bounds the distance
between the true and empirical distribution for any fixed confidence
level.

\begin{proposition}\longthmtitle{Nominal ambiguity radius~\cite[Corollary
		3.3]{DB-JC-SM:21-tac}} \label{prop:nominal:radius}
	Consider a sequence $\{X_i\}_{i\in\bN}$ of i.i.d. $\Rat{d}$-valued
	random variables with a compactly supported distribution $\mu$. Then
	for any $p\ge 1$, $N\ge 1$, and confidence $1-\beta$ with
	$\beta\in(0,1)$, we have $\bP(W_p(\mu^N,\mu) \le
	\eps_N(\beta,\rho))\ge 1-\beta$, where
	\begin{align} 
		\eps_{N}(\beta,\rho) & :=
		\begin{cases}  
			\left(\frac{\ln(C\beta^{-1})}{c}\right)^{\frac{1}{2p}}\frac{\rho}{N^{\frac{1}{2p}}},
			& {\rm if}\; p>d/2,
			\\
			h^{-1}\left(\frac{\ln(C\beta^{-1})}{cN}\right)^{\frac{1}{p}}\rho,
			& {\rm if}\; p=d/2,
			\\
			\left(\frac{\ln(C\beta^{-1})}{c}\right)^{\frac{1}{d}}\frac{\rho}{N^{\frac{1}{d}}},
			& {\rm if}\; p<d/2,
		\end{cases} \label{epsN:dfn}
	\end{align}
	$\mu^{N}:=\frac{1}{N}\sum_{i=1}^N\delta_{X_i}$, 
	$\rho:=\frac{1}{2}\diameter{\infty}{\supp{\mu}}$,  
	$h(x):=\frac{x^2}{(\ln(2+1/x))^2}$, $x>0$, and the constants $C$
	and $c$ depend only on $p$ and $d$.
\end{proposition}   

This result shows how the nominal ambiguity radius depends on the size
of the distribution's support, the confidence level, and the number of
samples, and is based on recent concentration of measure inequalities
from~\cite{NF-AG:15}. The determination of the constants $C$ and $c$
in \eqref{epsN:dfn}
for the whole spectrum of data dimensions $d$ and Wasserstein
exponents $p$ is a particularly cumbersome task. 
In Section~\ref{subsec:explicit:contants} we provide some
alternative concentration of measure results and use them to obtain
explicit formulas for these constants when $d>2p$.


\subsection{Noise ambiguity  
	radius}\label{subsec:noise:radius}

In this section, we quantify the noise ambiguity radius
$\widehat\eps_N(\subscr{\beta}{ns})$ for any prescribed confidence
$1-\subscr{\beta}{ns}$. We first give a result that uniformly bounds
the distance between the true- and estimator-based-empirical
distributions with prescribed confidence for all values of the initial
condition and the internal noise from the set
$B_\infty^{Nd}(\rho_{\xi_0})\times B_\infty^{N\ell q}(\rho_w)$, which
contains the support of their joint distribution (and hence all their
possible realizations). For the results of this section, the initial
condition and the internal noise are interpreted as deterministic
quantities, as discussed above. %
%
%
%

\begin{lemma}\longthmtitle{Distance between true- \&
		estimator-based-empirical     distribution}\label{lemma:empirical:vs:true}
	Let $(\bm z,\bm \omega)\in B_\infty^{Nd}(\rho_{\xi_0})\times
	B_\infty^{N\ell q}(\rho_w)$ and consider the discrete distribution
	$P_{\xi_\ell}^N\equiv P_{\xi_\ell}^N(\bm z,\bm \omega)$ and the
	empirical distribution $\widehat P_{\xi_\ell}^N\equiv \widehat
	P_{\xi_\ell}^N(\bm z,\bm \omega,\bm v_{\ell})$, where $\bm v_\ell$
	is the measurement noise of the realizations.  Then,
	\begin{subequations}
		\begin{align}
			W_p(\widehat P_{\xi_\ell}^N,P_{\xi_\ell}^N) & \le
			2^{\frac{p-1}{p}} \fM_w +2^{\frac{p-1}{p}} \Big(\frac{1}{N}
			\sum_{i=1}^N(\fE^i)^p \Big)^\frac{1}{p}, \quad
			\text{where} \label{empirical:true:vs:estimated:bound1}
			\\
			\fM_w & :=
			\sqrt{d}\|\Psi_\ell\|\rho_{\xi_0}+\sqrt{q}\sum_{k=1}^\ell
			\|\Psi_{\ell,\ell-k+1}G_{\ell-k}\|\rho_w, \label{frakMw}
			\\
			\fE^i & \equiv \fE(\bm v^i):= \sum_{k=1}^\ell
			\|\Psi_{\ell,\ell-k+1}K_{\ell-k}\|\|v_{\ell-k}^i\|_1. \label{frak:Eis}
		\end{align}
	\end{subequations}
\end{lemma}

The next result gives bounds for the norms of the random
variables~$\fE^i$ in Lemma~\ref{lemma:empirical:vs:true}.

\begin{lemma}\longthmtitle{Orlicz- \& $L^p$-norm bounds for
		$\fE^i$}\label{lemma:Orlicz:bounds}
	The random variables $\fE^i$ in~\eqref{frak:Eis} satisfy
	\begin{subequations}
		\begin{align}
			\|\fE^i\|_p\le\fM_v & :=M_vr\sum_{k=1}^\ell
			\|\Psi_{\ell,\ell-k+1}K_{\ell-k}\|, \label{frakMv}
			\\
			\|\fE^i\|_{\psi_p}\le\fC_v & := 
			C_vr\sum_{k=1}^\ell\|\Psi_{\ell,\ell-k+1}K_{\ell-k}\| ,
			\label{frakCv}
			\\
			\|\fE^i\|_p\ge\fm_v & :=m_vr^{\frac{1}{p}}\bigg(\sum_{k=1}^\ell
			\|\Psi_{\ell,\ell-k+1}K_{\ell-k}\|^p\bigg)^{\frac{1}{p}}, \label{frakmv}
		\end{align}
	\end{subequations}
	with $m_v$, $M_v$, and $C_v$ as given in \textbf{H4}.
\end{lemma}

The proofs of both results above are given in the Appendix. We further
rely on the following concentration of measure result around the mean
of nonnegative independent random variables, whose proof is also in
the Appendix, to bound the term $\big(\frac{1}{N} \sum_{i=1}^N
(\fE^i)^p \big)^\frac{1}{p}$, and control the Wasserstein distance
between the true- and the estimator-based-empirical distribution.

\begin{proposition}\longthmtitle{Concentration around $p$th mean}
	\label{prop:concentration:around:mean}
	Let $X_1,\ldots,X_N$ be scalar, nonnegative, independent random
	variables with finite $\psi_p$ norm and $\bE[X_i^p]=1$. Then,
	\begin{align}
		\bP\bigg(\bigg(\frac{1}{N}\sum_{i=1}^NX_i^p\bigg)^{\frac{1}{p}}-1\ge
		t\bigg) \le
		2\exp\Big(-\frac{c'N}{R^2}\alpha_p(t)\Big), \label{norm:concentration}
	\end{align} 
	for every $t\ge 0$, with $c'=1/10$,
	$R:=\max_{i\in[1:N]}\|X_i\|_{\psi_p}+1/\ln 2$, and
	\begin{align} \label{function:alphap} \alpha_p(s):=
		\begin{cases}
			s^2,\;\textup{if}\;s\in[0,1],
			\\
			s^p,\;\textup{if}\;s\in(1,\infty).
		\end{cases}
	\end{align}
\end{proposition}

Combining the results above,
we obtain the main result of this section regarding the ambiguity center
difference.

\begin{proposition}\longthmtitle{Distance guarantee between true- 
		\& estimator-based-empirical distribution}\label{prop:hatepsN:guarantees} 
	Consider a confidence $1-\subscr{\beta}{ns}$ and let
	\begin{align}
		&
		\widehat\eps_N(\subscr{\beta}{ns}):=2^{\frac{p-1}{p}}\bigg(\fM_w+\fM_v+\fM_v
		\alpha_p^{-1}\bigg(\frac{\fR^2}{c'N}
		\ln\frac{2}{\subscr{\beta}{ns}}\bigg)\bigg), \label{hat:epsN}
	\end{align}
	with $\fM_w$, $\fM_v$ given by \eqref{frakMw}, \eqref{frakMv}, 
	\begin{align}
		\fR :=\fC_v/\fm_v+1/\ln 2, \label{constant:frakR}
	\end{align}
	and $\fC_v$, $\fm_v$ as in \eqref{frakCv}, \eqref{frakmv}. Then, 
	for all $(\bm z,\bm \omega)\in
	B_\infty^{Nd}(\rho_{\xi_0})\times B_\infty^{N\ell q}(\rho_w)$, we
	have
	\begin{align}\label{noise:guarantee}
		\bP\big(W_p(\widehat P_{\xi_{\ell}}^N(\bm z,\bm \omega,\bm
		v_{\ell}), P_{\xi_{\ell}}^N(\bm z,\bm \omega)) \le \widehat
		\eps_N(\subscr{\beta}{ns})\big)\ge 1-\subscr{\beta}{ns}.
	\end{align}
\end{proposition}  
\begin{proof}
	For each $i$, 
	the random variable $X_i:=\fE^i/{\|\fE^i\|_p}$ satisfies $\|X_i\|_p=1$. 
	Thus, we obtain from Proposition~\ref{prop:concentration:around:mean} that
	\begin{align*}
		\bP\bigg(\bigg(\frac{1}{N}\sum_{i=1}^N\bigg(\frac{\fE^i}{{\|\fE^i\|_p}}\bigg)^p\bigg)
		^{\frac{1}{p}}-1\ge t\bigg)\le 
		2\exp\Big(-\frac{c'N}{R^2}\alpha_p(t)\Big),
	\end{align*}
	where $R =
	\max_{i\in[1:N]}\big\|\fE^i/{\|\fE^i\|_p}\big\|_{\psi_p}+1/\ln2$. From
	\eqref{frakCv}, \eqref{frakmv}, and \eqref{constant:frakR}, we deduce
	$\fR\ge R$, and thus,
	\begin{align*}
		& 
		\bP\bigg(\bigg(\frac{1}{N}\sum_{i=1}^N\bigg(\frac{\fE^i}{{\|\fE^i\|_p}}\bigg)^p\bigg)
		^{\frac{1}{p}}-1\ge t\bigg)\le 
		2\exp\Big(-\frac{c'N}{\fR^2}\alpha_p(t)\Big).
	\end{align*}
	Now, it follows from~\eqref{frakMv} that 
	\begin{align*}
		\fM_v\bigg(\frac{1}{N}\sum_{i=1}^N
		\bigg(\frac{\fE^i}{{\|\fE^i\|_p}}\bigg)^p\bigg)^{\frac{1}{p}}-\fM_v\ge
		\bigg(\frac{1}{N}\sum_{i=1}^N(\fE^i)^p\bigg) ^{\frac{1}{p}}-\fM_v.
	\end{align*}
	Thus, we deduce
	\begin{align*}
		& 
		\bP\bigg(\bigg(\frac{1}{N}\sum_{i=1}^N(\fE^i)^p\bigg)^{\frac{1}{p}}-\fM_v	
		\ge\fM_v t\bigg) \le 2\exp\Big(-\frac{c'N}{\fR^2}\alpha_p(t)\Big),
	\end{align*} 
	or, equivalently, that 
	\begin{align}
		\bP\bigg(\bigg(\frac{1}{N}\sum_{i=1}^N(\fE^i)^p\bigg)^{\frac{1}{p}}\ge 
		\fM_v+s\bigg) \le 2\exp\Big(-\frac{c'N}{\fR^2}\alpha_p 
		\Big(\frac{s}{\fM_v}\Big)\Big). \label{deviation:from:mean}
	\end{align} 
	To establish~\eqref{noise:guarantee}, it suffices by
	Lemma~\ref{lemma:empirical:vs:true} to show that
	\begin{align*}
		& \bP\bigg( 2^{\frac{p-1}{p}}\fM_w+2^{\frac{p-1}{p}}\Big(\frac{1}{N} 
		\sum_{i=1}^N(\fE^i)^p\Big)^\frac{1}{p} \le \widehat 
		\eps_N(\subscr{\beta}{ns})\bigg)\ge 1-\subscr{\beta}{ns}.
	\end{align*}
	By the definition of $\widehat \eps_N$ and exploiting that it is strictly 
	decreasing with $\subscr{\beta}{ns}$, it suffices to prove
	\begin{align*}
		\bP\bigg(\Big(\frac{1}{N}\sum_{i=1}^N(\fE^i)^p\Big)^\frac{1}{p}< \fM_v 
		+\fM_v 
		\alpha_p^{-1}\Big(\frac{\fR^2}{c'N}\ln\frac{2}{\subscr{\beta}{ns}}\Big)\bigg)
		\ge 1-\subscr{\beta}{ns}.
	\end{align*}  
	Setting $\tau=\alpha_p^{-1}\big(\frac{\fR^2}{c'N}\ln 
	\frac{2}{\subscr{\beta}{ns}}\big)$, we equivalently need to show
	\begin{align*}
		\bP\bigg(\Big(\frac{1}{N}\sum_{i=1}^N(\fE^i)^p\Big)^\frac{1}{p}
		\ge \fM_v+\tau\fM_v\bigg)\le \subscr{\beta}{ns},
	\end{align*}
	which follows by \eqref{deviation:from:mean} with $s=\tau \fM_v$. 
\end{proof}

\subsection{Overall ambiguity set}

Here we combine the results from Sections~\ref{subsec:nominal:radius}
and~\ref{subsec:noise:radius} to obtain the ambiguity set of the state
distribution in the following result.

\begin{theorem}\longthmtitle{Ambiguity set under noisy dynamics
		\& observations}\label{thm:ambiguity:set:noise}
	Consider data collected from $N$ realizations of
	system~\eqref{eq:data} in accordance to
	Assumptions~\ref{samplig:assumption}
	and~\ref{assumption:distributiuon:class}, a confidence $1-\beta$,
	and let $\subscr{\beta}{nom}, \subscr{\beta}{ns}\in(0,1)$
	satisfying~\eqref{confidence:decomposition}.  Then, the
	guarantee~\eqref{ambiguity:set:confidence} holds, where
	$\psi_N(\beta)$ is given in~\eqref{psi:dfn} and its components
	$\eps_N(\subscr{\beta}{nom})\equiv
	\eps_N(\subscr{\beta}{nom},\rho_{\xi_\ell})$ and
	$\widehat\eps_N(\subscr{\beta}{ns})$ are given by~\eqref{epsN:dfn}
	and~\eqref{hat:epsN}, respectively, with
	\begin{align} \label{rho:ell}
		\rho_{\xi_\ell}:=\sqrt{d}\|\Phi_\ell\|\rho_{\xi_0}
		+\sqrt{q}\sum_{k=1}^\ell\|\Phi_{\ell,\ell-k+1}G_{\ell-k}\|\rho_w.
	\end{align} 
\end{theorem}
\begin{proof}
	Due to \eqref{psi:dfn} and the triangle inequality for $W_p$,
	\begin{align*}
		& \{W_p(\widehat P_{\xi_{\ell}}^N,P_{\xi_{\ell}})\le\psi_N(\beta)\} 
		\supset 
		\{W_p(\widehat P_{\xi_{\ell}}^N,P_{\xi_{\ell}}^N)\le \widehat 
		\eps_N(\subscr{\beta}{ns})\} 
		 \\
		 & \hspace{20em}
		\cap\{W_p(P_{\xi_{\ell}}^N, 
		P_{\xi_{\ell}})\le\eps_N(\subscr{\beta}{nom},\rho_{\xi_\ell})\}. 
	\end{align*} 
	Thus, to show \eqref{ambiguity:set:confidence}, it suffices to show that 
	\begin{align} 
		\bE\Big[\ones_{\{W_p(\widehat P_{\xi_{\ell}}^N,P_{\xi_{\ell}}^N)-
			\widehat \eps_N(\subscr{\beta}{ns})\le 0\}}
		\times\ones_{\{W_p(P_{\xi_{\ell}}^N,
			P_{\xi_{\ell}})-\eps_N(\subscr{\beta}{nom},\rho_{\xi_\ell})\le
			0\}}\Big]\ge 1-\beta. \label{total:confidence:condition}
	\end{align} 
	We therefore exploit Lemma~\ref{lemma:expectation:inequality} with
	the random variable $X\equiv(\bm\xi_0,\bm w_{\ell})$, taking values
	in the compact set $K\equiv B_\infty^{Nd}(\rho_{\xi_0})\times
	B_\infty^{N\ell q}(\rho_w)$, the random variable $Y\equiv\bm
	v_\ell \in \Rat{N\ell r}$, and $g(X,Y)\equiv g(\bm{\xi_0},\bm
	w_{\ell},\bm v_{\ell})$, where
	\begin{align*}
		g(\bm{\xi_0},\bm w_{\ell},\bm v_{\ell}) & :=
		\ones_{\{W_p(P_{\xi_{\ell}}^N(\bm{\xi_0},\bm
			w_{\ell}),P_{\xi_{\ell}})-\eps_N(\subscr{\beta}{nom},\rho_{\xi_\ell})\le
			0\}}
		 \\
		 & \quad
		\times\ones_{\{W_p(\widehat P_{\xi_{\ell}}^N(\bm{\xi_0},\bm
			w_{\ell},\bm v_{\ell}),P_{\xi_{\ell}}^N(\bm{\xi_0},\bm
			w_{\ell}))-\widehat\eps_N(\subscr{\beta}{ns})\le 0\}}.
	\end{align*}
	Due to~\eqref{noise:guarantee}, $\bE\Big[\ones_{\{W_p(\widehat 
		P_{\xi_\ell}^N(\bm z,\bm \omega,\bm v_{\ell}), P_{\xi_{\ell}}^N(\bm z,\bm
		\omega))-\widehat\eps_N(\subscr{\beta}{ns})\le 0\}}\Big]\ge
	1-\subscr{\beta}{ns}$ for any $x=(\bm z,\bm\omega)\in K$
	and thus $ \bE
	[g(x,Y)]\ge\ones_{\{W_p(P_{\xi_\ell}^N(x),P_{\xi_\ell})
		-\eps_N(\subscr{\beta}{nom},\rho_{\xi_\ell})\le 0\}}\times
	(1-\subscr{\beta}{ns})=:k(x) $, for all $x\in K$. Hence, since
	$X\equiv(\bm\xi_0,\bm w_{\ell})$ and $Y\equiv\bm v_\ell$ are
	independent by \textbf{H2}, we deduce from
	Lemma~\ref{lemma:expectation:inequality} that
	\begin{align*}
		\bE [g(X,Y)] & \ge\bE\Big[\ones_{\{W_p(P_{\xi_\ell}^N
			(\bm\xi_0,\bm w_\ell),
			P_{\xi_\ell})-\eps_N(\subscr{\beta}{nom},\rho_{\xi_\ell})\le
			0\}}(1-\subscr{\beta}{ns})\Big]
		\\
		& = (1-\subscr{\beta}{ns})\bP(W_p(P_{\xi_\ell}^N(\bm\xi_0,\bm
		w_\ell),P_{\xi_\ell})\le\eps_N(\subscr{\beta}{nom},\rho_{\xi_\ell})).
	\end{align*}
	From \eqref{pushed:forward:ditributiuon} and \textbf{H3} in
	Assumption~\ref{assumption:distributiuon:class}, it follows that
	$P_{\xi_\ell}$ is supported on the compact set 
	$B_\infty^d(\rho_{\xi_\ell})$  with 
	$\diameter{\infty}{B_\infty^d(\rho_{\xi_\ell})} =
	2\rho_{\xi_\ell}$ and $\rho_{\xi_\ell}$ given in~\eqref{rho:ell}. In
	addition, due to \textbf{H1} and \textbf{H2} in
	Assumption~\ref{assumption:distributiuon:class} the random states
	$\xi_{\ell}^i$ in the empirical distribution $P_{\xi_\ell}^N
	(\bm\xi_0,\bm w_\ell)=\frac{1}{N}\sum_{i=1}^N\delta_{\xi_{\ell}^i}$
	are i.i.d.. Thus, 
	we get from Proposition~\ref{prop:nominal:radius} that
	$\bP(W_p(P_{\xi_\ell}^N (\bm\xi_0,\bm w_\ell),P_{\xi_\ell}) \le
	\eps_N(\subscr{\beta}{nom},\rho_{\xi_\ell}))\ge
	1-\subscr{\beta}{nom}$, which implies
	$
	\bE [g(X,Y)]\ge (1-\subscr{\beta}{ns})(1-\subscr{\beta}{nom})=1-\beta.
	$ 
	Finally,~\eqref{total:confidence:condition} follows from this and
	the definition of~$g$.
\end{proof}

With this result at hand, we deduce from the
expressions~\eqref{epsN:dfn} and \eqref{noise:guarantee} for the
components of the ambiguity radius that it decreases as we exploit a
larger number $N$ of independent trajectories and relax our confidence
choices, i.e., reduce $1-\subscr{\beta}{nom}$ and
$1-\subscr{\beta}{ns}$. Notice further that no matter how many
trajectories we use, the noise ambiguity radius decreases to a
strictly positive value.
%
It is also worth to observe that
$\psi_N$ generalizes the nominal ambiguity radius $\eps_N$ in the DRO
literature (even when dynamic random variables are
considered~\cite{DB-JC-SM:21-tac}) and reduces to $\eps_N$ in the
noise-free case where~$\widehat\eps_N=0$.

Drawing conclusions about how the ambiguity radius behaves as we
simultaneously allow the horizon $[0:\ell]$ and the number $N$ of
sampled trajectories to increase is a more delicate matter. The value
of the nominal component depends essentially on $N$ and the support of
the distribution at $\ell$, with the latter in turn depending on the
system's stability properties and the support of the initial condition
and internal noise distributions. On the other hand, the noise
component depends on $N$ and the quality of the estimation error. We
quantify in the next section how the latter guarantees uniform
boundedness of the noise radius under detectability-type assumptions.

\begin{remark}\longthmtitle{Positive lower bound of the noise radius}
	{\rm The positive lower bound $2^{\frac{p-1}{p}}(\fM_w+\fM_v)$ on
		the noise radius in~\eqref{hat:epsN} represents in general a
		fundamental limitation for the ambiguity set accuracy, which is
		independent of the number $N$ of estimated state samples. This is
		because the bound is related to the size of the state estimation
		error, which persists under the presence of noise and may further
		grow in time if there is no system detectability.}  \oprocend
\end{remark}

\begin{remark}\longthmtitle{Optimal radius selection} {\rm Once a
		desired confidence level $1-\beta$ and the number of independent
		trajectories $N$ are fixed, we can optimally select the ambiguity
		radius by minimizing the function
		\begin{align*}
			\subscr{\beta}{nom}\mapsto
			\psi_N(\subscr{\beta}{nom})\equiv\eps_N(\subscr{\beta}{nom})+\widehat\eps_N
			((\beta-\subscr{\beta}{nom})/(1-\subscr{\beta}{nom})),
		\end{align*}
		where we have taken into account the
		constraint~\eqref{confidence:decomposition} between the nominal
		and the noise confidence. This function is non-convex, but
		one-dimensional, and its minimizer is in the interior of the
		interval~$(0,\beta)$, so its optimal value can be approximated
		with high accuracy.}  \oprocend
\end{remark}

\subsection{Uncertainty quantification over bounded time horizons}

In this section we discuss how the guarantees can be extended to
scenarios where an ambiguity set is built over a finite-time horizon
instead of a single instance $\ell$. In this case we assume that
samples are collected over the time window $[0:\ell_2]$ and we seek to
build an ambiguity set about the state distribution along
$[\ell_1:\ell_2]$, with $0\le\ell_1\le\ell_2$. We distinguish between
two ambiguity set descriptions depending on the way the associated
probabilistic guarantees are obtained.  In the first, we directly
build an ambiguity set for the probability distribution of the random
vector $\bm
\xi_{\bm\ell}:=(\xi_{\ell_1},\ldots,\xi_{\ell_2})\in\Rat{\widetilde\ell
	d}$ with $\bm\ell:=(\ell_1,\ldots,\ell_2)$ and
$\widetilde\ell=\ell_2-\ell_1+1$, comprising of all states over the
interval of interest and using the concentration of measure result of
Proposition~\ref{prop:nominal:radius} for $\widetilde\ell
d$-dimensional random variables. This has the drawback that the
ambiguity radius decays slowly with the number of trajectories due to
the high dimension of $\bm\xi_{\bm\ell}$. The other description
derives an ambiguity set about the state distribution
$P_{\xi_{\ell_1}}$ at time $\ell_1$ with prescribed confidence, and
propagates it under the dynamics while taking into account the
possible values of the internal noise. We also present sharper results
for the cases when the internal noise sequence is known. 
The first ambiguity set description is provided by the
following analogue of Theorem~\ref{thm:ambiguity:set:noise}.

\begin{theorem}\longthmtitle{Ambiguity set over a bounded time
		horizon}
	\label{thm:ambiguity:over:horizon}
	Consider output data collected from $N$ realizations of
	system~\eqref{eq:data} over the interval $[0:\ell_2]$ and let
	Assumption~\ref{samplig:assumption} hold. Pick a confidence
	$1-\beta$, let $\subscr{\beta}{nom},\subscr{\beta}{ns}\in(0,1)$
	satisfying~\eqref{confidence:decomposition}, and consider the
	bounded-horizon estimator empirical distribution
	\begin{align*}
		\widehat P_{\bm\xi_{\bm\ell}}^N:=\frac{1}{N}\sum_{i=1}^N 
		\delta_{\widehat{\bm\xi}_{\bm\ell}^i}
	\end{align*}
	over the horizon $[\ell_1:\ell_2]$, where
	$\widehat{\bm\xi}_{\bm\ell}^i:=
	(\widehat{\xi}_{\ell_1}^i,\ldots,\widehat{\xi}_{\ell_2}^i)\in\Rat{\widetilde\ell
		d}$ and each $\widehat{\xi}_{\ell}^i$ is given by the observer
	\eqref{observer:dynamics}. Then
	\begin{align}
		\bP(W_p(\widehat P_{\bm\xi_{\bm\ell}}^N,P_{\bm\xi_{\bm\ell}}) 
		\le\psi_N(\beta))\ge 1-\beta
	\end{align} 
	holds, where $\bm\xi_{\bm\ell}:=(\xi_{\ell_1},\ldots,\xi_{\ell_2})$
	and $\psi_N(\beta)$ is given in~\eqref{psi:dfn}. The nominal
	component $\eps_N(\subscr{\beta}{nom})\equiv
	\eps_N(\subscr{\beta}{nom},\rho_{\bm\xi_{\bm\ell}})$ is given
	by~\eqref{epsN:dfn} (with $d$ in the expression substituted by
	$\widetilde\ell d$)
	\begin{align} \label{rhoxiell:stacked:states}
		\rho_{\bm\xi_{\bm\ell}}:=\max_{\ell\in[\ell_1:\ell_2]}\bigg\{\sqrt{d}\|\Phi_\ell\|\rho_{\xi_0}
		+\sqrt{q}\sum_{k=1}^\ell\|\Phi_{\ell,\ell-k+1}G_{\ell-k}\|\rho_w\bigg\},
	\end{align} 
	whereas $\widehat\eps_N(\subscr{\beta}{ns})$ is given as 
	\begin{align*}
		\widehat\eps_N(\subscr{\beta}{ns}) & :=2^{\frac{p-1}{p}}
		\bigg(\widetilde{\fM}_w+\widetilde{\fM}_v+\widetilde{\fM}_v
		\alpha_p^{-1}\bigg(\frac{\widetilde\fR^2}{c'N}
		\ln\frac{2}{\subscr{\beta}{ns}}\bigg)\bigg), \quad \text{with}
		\\
		\widetilde\fM_w & := \sum_{\ell=\ell_1}^{\ell_2}\fM_w(\ell),\qquad
		\widetilde\fM_v := \sum_{\ell=\ell_1}^{\ell_2}\fM_v(\ell), \\
		\widetilde\fR & :=
		\frac{\widetilde\fC_v}{\widetilde\fm_v}+\frac{1}{\ln 2}, \qquad
		\widetilde\fC_v:=\sum_{\ell=\ell_1}^{\ell_2}\fC_v(\ell),\qquad
		\widetilde\fm_v:=\sum_{\ell=\ell_1}^{\ell_2}\fm_v(\ell),
	\end{align*}
	and $\fM_w(\ell)\equiv\fM_w$, $\fM_v(\ell)\equiv\fM_v$, 
	$\fC_v(\ell)\equiv\fC_v$, and $\fm_v(\ell)\equiv\fm_v$, as given by 
	\eqref{frakMw}, \eqref{frakMv}, \eqref{frakCv}, and \eqref{frakmv}, 
	respectively. 
\end{theorem}

The proof of this result follows the argumentation employed for the
proof of Theorem~\ref{thm:ambiguity:set:noise} (a sketch can be found
in the Appendix). For the second ambiguity set description we use a 
pointwise-in-time approach. To this
end, we build a family of ambiguity balls so that under the same
confidence level the state distribution at each time instant of the
horizon lies in the associated ball, i.e.,
\begin{align}\label{pontwise:ambiguity:sets}
	\bP\big(P_{\xi_{\ell}}\in \B_{\psi_{N,\ell}}(\widetilde 
	P_{\xi_\ell}^N)\;\forall\ell\in[\ell_1:\ell_2]\big)\ge 1-\beta,
\end{align}
where ${\B}_{\psi_{N,\ell}}(\widetilde P_{\xi_\ell}^N\big):=
\{P\in\P_p(\Rat{d})\,|\,W_p(P,\widetilde
P_{\xi_\ell}^N)\le\psi_{N,\ell}\}$ {and $\widetilde P_{\xi_\ell}^N$ is
	the center of the ball}. This is well suited for stochastic
optimization problems that have a separable structure with respect to
the stochastic argument across different time instances, i.e.,
problems of the form
\begin{align*}
	\inf_{x\in\X}\bE\big[f_1(x,\xi_{\ell_1})+\cdots
	+f_{\widetilde\ell}(x,\xi_{\ell_2})\big].
\end{align*}       
The pointwise ambiguity sets are quantified in the following result.
%

\begin{theorem}\longthmtitle{Pointwise ambiguity sets over a bounded
		time horizon}\label{thm:pointwise:ambiguity:set}
Let the assumptions of Theorem~\ref{thm:ambiguity:over:horizon} hold, 
assume that the internal noise sequence $w_\ell$  is independent (also of 
the initial state), $P_{w_\ell}\in\P_p(\Rat{d})$ for 
$\ell\in[\ell_1:\ell_2]$, i.e., it is not necessarily compactly supported, 
and consider either of the following two cases for its distribution 
when $\ell\in[\ell_1:\ell_2]$:
\begin{enumerate}
\item $P_{w_\ell}$ is not known and 
$\bE\big[\|w_\ell\|^p\big]^{\frac{1}{p}}\le q_w$.
\item $P_{w_\ell}$ is known.
\end{enumerate}
Then, for any confidence $1-\beta$, and
$\subscr{\beta}{nom},\subscr{\beta}{ns}\in(0,1)$
satisfying~\eqref{confidence:decomposition},
\eqref{pontwise:ambiguity:sets} holds, with $\widetilde
P_{\xi_{\ell_1}}^N:=\widehat P_{\xi_{\ell_1}}^N$ and
$\psi_{N,\ell_1}$ as given by Theorem~\ref{thm:ambiguity:set:noise}
(for $\ell\equiv\ell_1$), and $\widetilde P_{\xi_{\ell}}^N$,
$\psi_{N,\ell}$, $\ell\in[\ell_1+1:\ell_2]$ defined as follows for
the respective two cases above:
\begin{enumerate}
	\item The ambiguity set center is $\widetilde
	P_{\xi_\ell}^N:=\frac{1}{N}\sum_{i=1}^N\delta_{\widetilde\xi_\ell^i}$
	with
	$\widetilde\xi_\ell^i:=\Phi_{\ell,\ell_1}\widehat\xi_{\ell_1}$
	and the radius is given recursively by
	$\psi_{N,\ell}:=\|A_{\ell-1}\|\psi_{N,\ell-1}+q_w$.
	\item The ambiguity set center is $\widetilde
	P_{\xi_\ell}^N:=\big((A_{\ell-1})_\#\widetilde
	P_{\xi_{\ell-1}}^N\big)\star P_{w_{\ell-1}}$ and the radius is
	$\psi_{N,\ell}:=\|A_{\ell-1}\|\cdots\|A_{\ell_1}\|\psi_{N,\ell_1}$.
\end{enumerate}
\end{theorem}

%
%

In our technical approach, we use the next result, whose proof is
given in the Appendix. The result examines what happens to the
Wasserstein distance between the distributions of two random variables
when other random variables are added.

\begin{lemma}\longthmtitle{Wasserstein
		distance under convolution}\label{lemma:Wasserstein:convolution}
	Given $p\ge 1$ and distributions $P_1,P_2,Q\in\P_p(\Rat{d})$, it
	holds that
	$
	W_p(P_1,P_2)\le W_p(P_1\star Q,P_2\star Q).
	$ 
  Also, if it holds that $\big(\int_{\Rat{d}}\|x\|^pQ(dx)\big)^{\frac{1}{p}}\le q$, 
then $W_p(P_1,P_2\star Q)\le W_p(P_1,P_2)+q$.      
\end{lemma} 

\begin{proof}[Proof of Theorem~\ref{thm:pointwise:ambiguity:set}]
	The proof is carried out by induction on
	$\ell\in[\ell_1:\ell_2]$. In particular, it suffices to establish
	that
	\begin{align} \label{pointwise:inclusion:induction}
		W_p\big(P_{\xi_{\ell_1}},\widetilde 
		P_{\xi_{\ell_1}}^N\big)\le\psi_{N,\ell_1}\Longrightarrow 
		W_p\big(P_{\xi_{\ell'}},\widetilde 
		P_{\xi_{\ell'}}^N\big)\le\psi_{N,\ell'} 
		\;\forall\ell'\in[\ell_1:\ell].
	\end{align}
	Note that from Theorem~\ref{thm:ambiguity:set:noise},
	$\bP\big(P_{\xi_{\ell_1}}\in \B_{\psi_{N,\ell_1}}(\widetilde
	P_{\xi_{\ell_1}}^N)\big)\ge 1-\beta$. From
	\eqref{pointwise:inclusion:induction}, this also implies that $
	\bP\big(P_{\xi_{\ell'}}\in \B_{\psi_{N,\ell'}}(\widetilde
	P_{\xi_{\ell'}}^N)\;\forall\ell'\in[\ell_1:\ell]\big)\ge 1-\beta$,
	establishing validity of the result for $\ell\equiv\ell_2$.
	
	For $\ell\equiv\ell_1$, the induction
	hypothesis~\eqref{pointwise:inclusion:induction} is a
	tautology. Next, assuming that it is true for certain
	$\ell\in[\ell_1:\ell_2-1]$, we show that it also holds for
	$\ell+1$. Hence it suffices to show that if
	$W_p\big(P_{\xi_\ell},\widetilde P_{\xi_\ell}^N\big)\le\psi_{N,\ell}
	$ then also $W_p\big(P_{\xi_{\ell+1}},\widetilde
	P_{\xi_{\ell+1}}^N\big)\le\psi_{N,\ell+1}$ for both cases (i)
	and~(ii). Consider first (i) and note that then the ambiguity set
	center at $\ell+1$ satisfies
	\begin{align*}
		\widetilde P_{\xi_{\ell+1}}^N
		=\frac{1}{N}\sum_{i=1}^N\delta_{\widetilde\xi_{\ell+1}^i}
		=\frac{1}{N}\sum_{i=1}^N\delta_{A_\ell\widetilde\xi_\ell^i}=
		(A_\ell)_\#\widetilde P_{\xi_\ell}^N,
	\end{align*}
	where we have exploited that
	$\widetilde\xi_k^i\equiv\Phi_{k,\ell_1}\widehat\xi_{\ell_1}^i$ and
	the definition of $\Phi_{k,\ell_1}$ (for $k=\ell-1,\ell$) to derive
	the second equality. Using also the fact that
	$P_{\xi_{\ell+1}}=\big((A_\ell)_\#P_{\xi_\ell}\big) \star
	P_{w_\ell}$, we get from the second result of
	Lemma~\ref{lemma:Wasserstein:convolution} that
	\begin{align*}
		W_p(P_{\xi_{\ell+1}},\widetilde P_{\xi_{\ell+1}}^N) &
		=W_p\big(\big((A_\ell)_\#P_{\xi_\ell}\big)\star
		P_{w_\ell},(A_\ell)_\#\widetilde P_{\xi_\ell}^N\big) \\
		& \le
		W_p\big((A_\ell)_\#P_{\xi_\ell},(A_\ell)_\#\widetilde
		P_{\xi_\ell}^N\big)+q_w \\
		& \le \|A_\ell\|W_p\big(P_{\xi_\ell},\widetilde
		P_{\xi_\ell}^N\big)+q_w\le\|A_\ell\|\psi_{N,\ell}+q_w
		=\psi_{N,\ell+1}.
	\end{align*}
	Here we also used the fact that $W_p(f_\#P,f_\#Q)\le L W_p(P,Q)$ for
	any globally Lipschitz function $f:\Rat{d}\to\Rat{r}$ with Lipschitz
	constant $L$ in the second to last inequality (see e.g.,
	\cite[Proposition 7.16]{CV:03}).
	
	Next, we prove the induction hypothesis for (ii). Using
	Lemma~\ref{lemma:Wasserstein:convolution} and the definition of the
	ambiguity set center and radius,
	\begin{align*}
		W_p(P_{\xi_{\ell+1}},\widetilde P_{\xi_{\ell+1}}^N) &
		=W_p\big(\big((A_\ell)_\#P_{\xi_\ell}\big)\star P_{w_\ell},
		\big((A_\ell)_\#\widetilde P_{\xi_\ell}^N\big)
		\star P_{w_\ell}\big) \\ 
		& \le W_p\big((A_\ell)_\#P_{\xi_\ell},(A_\ell)_\#\widetilde
		P_{\xi_\ell}^N\big) \le \|A_\ell\|W_p\big(P_{\xi_\ell},\widetilde
		P_{\xi_\ell}^N\big) \\
		& \le\|A_\ell\|\psi_{N,\ell}
		=\|A_\ell\|\|A_{\ell-1}\|\cdots\|A_{\ell_1}\|\psi_{N,\ell_1}
		=\psi_{N,\ell+1},
	\end{align*}
	completing the proof.
\end{proof}

\section{Sufficient conditions for uniformly bounded noise ambiguity
	radii} \label{sec:uniform:noise:radius:bnd}

In this section we leverage Assumption~\ref{detectability:assumption} to 
establish that the noise ambiguity radius remains uniformly bounded as the 
sampling horizon increases. We first provide uniform bounds for the matrices 
involved in the system and observer error dynamics.

\begin{proposition}\longthmtitle{Bounds on system/observer
		matrices}\label{prop:fundamental:matrix:contraction}
	Under Assumption~\ref{detectability:assumption}, the gain matrices
	$K_k$ can be selected so that the following properties hold:
	\begin{enumerate}
		\item There exist $K_\star,K^\star,G^\star>0$ and
		$\Psi_s^\star>0$, $s\in\mathbb N_0$, so that $\|G_k\|\le G^\star$,
		$K_\star\le\|K_k\|\le K^\star$, and $\|\Psi_{k+s,k}\|\le
		\Psi_s^\star$ for all and $k\in\mathbb N_0$.
		\item There exists $s_0\in\mathbb N$ so that
		$\|\Psi_{k+s,k}\|\le\frac{1}{2}$ for all $k\in\mathbb N_0$ and
		$s\ge s_0$.
	\end{enumerate}
\end{proposition}
\begin{proof}
	Note that we only need to verify part \emph{(i)} for the
	time-varying case.  Since all $G_k$ are uniformly bounded, we
	directly obtain the bound $G^\star$. Let
	\begin{align*}
		K_k:=-A_k \Phi_{k,k-t-1}{\Obs_{k,k-t-1}^{-1}}\Phi_{k,k-t-1}^\top 
		H_k^\top,
	\end{align*}
	(for $k>t+1$) as selected in \cite[Page 574]{JBM-BDOA:80} (but with
	a minus sign at the front to get the plus sign in $F_k=A_k+K_kH_k$)
	and with the observability Gramian $\Obs_{k,k-t-1}$ as defined in
	Assumption~\ref{detectability:assumption}(ii). Then, the upper bound
	$K^\star$ follows from the fact that the system matrices are
	uniformly bounded combined with the uniform observability property
	of Assumption~\ref{detectability:assumption}, which implies that all
	$\Obs_{k,k-t-1}^{-1}$ are also uniformly bounded. On the other hand,
	the lower bound $K_\star$ follows from the assumption that the
	system matrices are uniformly bounded, which imposes a uniform lower
	bound on the smallest singular value of $\Obs_{k,k-t-1}^{-1}$, the
	uniform lower bound on the smallest singular value of $A_k$, hence,
	also on that of $\Phi_{k,k-t-1}$ and $\Phi_{k,k-t-1}^\top$, and the
	uniform lower bound on $\|H_k\|$ (all found in
	Assumption~\ref{detectability:assumption}).  Finally, the bounds
	$\Psi_s^\star$ follow from the uniform bounds for all $A_k$ and
	$H_k$ and the derived bound $K^\star$ for all~$K_k$.
	
	To show part \emph{(ii)}, assume first that
	Assumption~\ref{detectability:assumption}(i) holds, i.e., the system
	is time invariant and $(A,H)$ is detectable. Then, we can choose a
	nonzero gain matrix $K$ so that $F=A+KH$ is
	convergent~(cf. \cite[Theorem 31]{EDS:98}), namely
	$\lim_{s\to\infty}\|F^s\|=0$. Consequently, there is $s_0\in\bN$
	with $\|F^s\|\le\frac{1}{2}$ for all $s\ge s_0$ and the result
	follows by taking into account that $\Psi_{k+s,k}=F^s$. In case
	Assumption~\ref{detectability:assumption}(ii) holds, let
	\begin{align}\label{undisturbed:error:eqn}
		\widetilde e_{k+1}=F_k\widetilde e_k
	\end{align}  
	be the recursive noise-free version of the error
	equation~\eqref{error:formula}. Then, from \cite[Page
	577]{JBM-BDOA:80}, 
	there exists a quadratic time-varying Lyapunov function
	$V(k,\widetilde e):=\widetilde e^\top Q_k\widetilde e$ with each
	$Q_k$ being positive definite, $a_1,a_2>0$, $a_3\in(0,1)$, and
	$m\in\mathbb N$ so that
	\begin{subequations}
		\begin{align}
			a_1\le\lambda_{\min}(Q_k)\le\lambda_{\max}(Q_k) & \le a_2 
			\label{eigenvalue:sandwitch:bnds}\\
			V(k+m,\widetilde e_{k+m})-V(k,\widetilde e_k) &
			\le-a_3V(k,\widetilde e_k)
			\label{Lyapunov:decrease} 
		\end{align}  
	\end{subequations}
	for any $k$ and any solution of \eqref{undisturbed:error:eqn} with
	state 
	$\widetilde e_k$ at time $k$. Thus, we have that $\Psi_{k+m,m}^\top
	Q_{k+m}\Psi_{k+m,m}\preceq (1-a_3)Q_k $, and hence, by induction, 
	we get that $\Psi_{k+\nu m,m}^\top$ $\times Q_{k+\nu m}\Psi_{k+\nu m,m}\preceq
	(1-a_3)^\nu Q_k$, since
	\begin{align*}
		& \Psi_{k+(\nu+1)m,m}^\top Q_{k+(\nu+1)m}\Psi_{k+(\nu+1)m,m}
		 \\
		 &\qquad
		\! = \! \Psi_{k+m,k}^\top\Psi_{k+(\nu+1) m,k+m}^\top
		Q_{k+(\nu+1)m}\Psi_{k+(\nu+1)m,k+m}\Psi_{k+m,k}
		\\
		& \qquad\preceq (1-a_3)^\nu \Psi_{k+m,k}^\top Q_{k+m}\Psi_{k+m,k}
		\preceq (1-a_3)^{(\nu+1)}Q_k.
	\end{align*} 
	Next, pick $\widetilde e$ with $\|\widetilde e\|=1$ and
	$\|\Psi_{k+\nu m,m}\widetilde e\|=\|\Psi_{k+\nu m,m}\|$. Taking into 
	account that $\widetilde e^\top\Psi_{k+\nu m,k}^\top Q_{k+\nu 
		m}\Psi_{k+\nu 
		m,m}\widetilde e \le (1-a_3)^\nu\widetilde e^\top Q_k\widetilde e$, we 
	get $\lambda_{\min}(Q_{k+\nu m})\|\Psi_{k+\nu m,k}\widetilde e\|^2 \le
	(1-a_3)^\nu\lambda_{\max}(Q_k)$. Using~\eqref{eigenvalue:sandwitch:bnds},
	\begin{align}
		\|\Psi_{k+\nu m,k}\| & \le
		(1-a_3)^{\frac{\nu}{2}}\Big(\frac{a_2}{a_1}\Big)^{\frac{1}{2}}.
		\label{Psi:bound}
	\end{align}
	Now, select $\nu$ so that $
	(1-a_3)^{\frac{{\nu'}}{2}}({a_2}/{a_1})^{\frac{1}{2}} \le
	{1}/{(2\max_{s\in[1:m]}\Psi_s^\star)}$ for all $\nu'\ge\nu$. Let 
	$s_0:=\nu m$ and pick $s\ge s_0$. Then, $s=s_0'+m'$ for some 
	$s_0'=\nu'm$, $\nu'\ge\nu$, and $m'\in[0:m-1]$ and we get from 
	\eqref{Psi:bound}, part~(i), and the selection of $\nu$ that
	\begin{align*}
		\|\Psi_{k+sm,k}\| 
		& \!=\! \|\Psi_{k+{s_0'}+m',k+{s_0'}}\Psi_{k+{s_0'},k}\| 
		 \\
		 & 
		\!    \le \! \|\Psi_{k+{s_0'}+m',k+{s_0'}}\|\|\Psi_{k+{\nu'}m,k}\| 
		\! \le \! \Psi_{m'}^\star\frac{1}{2\max_{s\in[1:m]}\Psi_s^\star}
		\! \le \!
		\frac{1}{2},
	\end{align*}
	which establishes the result. 	  
\end{proof}

Based on this result and Assumption~\ref{detectability:assumption}
about the system's detectability/uniform observability properties, we
proceed to provide a uniform bound on the size of the noise radius for
arbitrarily long evolution horizons.

\begin{proposition}\longthmtitle{Uniform bounds for noise
		ambiguity radius}\label{prop:noise:radius:bounds}
	Consider data collected from $N$ realizations of
	system~\eqref{eq:data}, a confidence $1-\beta$ as
	in~\eqref{confidence:decomposition}, and let
	Assumptions~\ref{samplig:assumption},~\ref{assumption:distributiuon:class},
	and~\ref{detectability:assumption} hold. Then, there exist
	observer gain matrices $K_k$ so that the noise ambiguity radius
	$\widehat\eps_N$ in \eqref{hat:epsN} is uniformly bounded with
	respect to the sampling horizon size. In particular, there exists
	$\ell_0\in\bN$ so that, for each $\ell\ge\ell_0$,
	$\fM_w\equiv\fM_w(\ell)$, $\fM_v\equiv\fM_v(\ell)$, and
	$\fR\equiv\fR(\ell)$ given by \eqref{frakMw}, \eqref{frakMv}, and
	\eqref{constant:frakR}, are uniformly upper bounded~as
	\begin{align*}
		\fM_w & \le
		\frac{1}{2}\sqrt{d}\rho_{\xi_0}+{3}\sqrt{q}\sum_{j=0}^{\ell_0-1}\Psi_j^\star
		G^\star\rho_w,
		\\
		\fM_v & \le {3}M_vr\sum_{j=0}^{\ell_0-1}\Psi_j^\star K^\star, \qquad
		\fR \le
		{3}\frac{C_v}{m_v}r^{\frac{p-1}{p}}\frac{\sum_{j=0}^{\ell_0-1}\Psi_j^\star
			K^\star}{K_\star}.
	\end{align*}
\end{proposition}
\begin{proof}
	Consider gain matrices $K_k$ and the time $s_0$ as given in
	Proposition~\ref{prop:fundamental:matrix:contraction}, and let
	$\ell_0:=s_0$. Then, for any $\ell\ge\ell_0$, $\ell=n\ell_0+r'$
	with $0\le r'<\ell_0$ and we have
	\begin{align*}
		\sum_{k=1}^\ell & \|\Psi_{\ell,\ell-k+1}G_{\ell-k}\|\le
		\sum_{k=1}^{\ell}\|\Psi_{\ell,\ell-k+1}\|G^\star \\
		& =\bigg(\sum_{k=1}^{r'}\|\Psi_{\ell,\ell-k+1}\|
		+\sum_{k=r'+1}^{\ell}\|\Psi_{\ell,\ell-k+1}\|\bigg)G^\star\\
		&
		\le\bigg(\sum_{s=0}^{r'-1}\Psi_s^\star
		+\sum_{k=r'+1}^{n\ell_0+r'}\|\Psi_{n\ell_0+r',n\ell_0+r'-k+1}\|\bigg)G^\star
		\quad(k\mapsto(\nu-1)\ell_0+j+r')\\
		& 
		=\bigg(\sum_{s=0}^{r'-1}\Psi_s^\star
		+\sum_{\nu=1}^n\sum_{j=1}^{\ell_0}\|
		\Psi_{n\ell_0+r',(n-\nu)\ell_0+r'+\ell_0-j+1}\|\bigg)G^\star
		\quad(\ell_0+1-j\mapsto j) \\
		&
		\le\bigg(\sum_{s=0}^{r'-1}\Psi_s^\star
		+\sum_{\nu=1}^n\sum_{j=1}^{\ell_0}
		\|\Psi_{n\ell_0+r',(n-\nu+1)\ell_0+r'}\|
		\|\Psi_{(n-\nu)\ell_0+r'+\ell_0,(n-\nu)\ell_0+r'+j}\|\bigg)G^\star\\
		& =\bigg(\sum_{s=0}^{r'-1}\Psi_s^\star
		+\sum_{\nu=1}^n\|\Psi_{n\ell_0+r',(n-\nu+1)\ell_0+r'}\|
		\sum_{j=1}^{\ell_0}
		\|\Psi_{(n-\nu)\ell_0+r'+\ell_0,(n-\nu)\ell_0+r'+j}\|\bigg)G^\star\\
		&\le\bigg(\sum_{s=0}^{r'-1}\Psi_s^\star
		+\sum_{\nu=1}^n\bigg(\prod_{\kappa=1}^{\nu-1}
		\|\Psi_{(n+1-\kappa)\ell_0+r',(n-\kappa)\ell_0+r'}\|\bigg)
		\sum_{j=1}^{\ell_0}\Psi_{\ell_0-j}^\star\bigg)G^\star\\
		&\le\bigg(\sum_{s=0}^{\ell_0-1}\Psi_s^\star
		+\sum_{\nu=1}^n\Big(\frac{1}{2}\Big)^{\nu-1}
		\sum_{j=0}^{\ell_0-1}\Psi_j^\star\bigg)G^\star\le 3
		\sum_{j=0}^{\ell_0-1}\Psi_j^\star G^\star,
	\end{align*}
	where we have used $\sum_{\kappa=0}^{-1}\equiv
	\sum_{\kappa=1}^0\equiv 0$ and $\prod_{\kappa=1}^0\equiv 1$. From
	this and the fact that from
	Proposition~\ref{prop:fundamental:matrix:contraction},
	$\|\Psi_{\ell}\|\le\frac{1}{2}$ for all $\ell\ge\ell_0$, we get the
	upper bound for~$\fM_w$. The one for $\fM_v$ is obtained
	analogously. Finally, for~$\fR$, we obtain the same type of upper
	bound for $\fC_v$ as for $\fM_w$, and exploit
	Proposition~\ref{prop:fundamental:matrix:contraction}(i) to get the
	lower bound $\fm_v=m_vr^{\frac{1}{p}}
	\big(\sum_{k=1}^\ell\|\Psi_{\ell,\ell-k+1}K_{\ell-k}\|^p\big)^{\frac{1}{p}}
	\ge m_vr^{\frac{1}{p}} \|\Psi_{\ell,\ell}K_{\ell-1}\|\ge
	m_vr^{\frac{1}{p}} K_\star$, which is independent of~$\ell$.
\end{proof}

\begin{remark}\longthmtitle{Noise ambiguity radius for time-invariant
		systems} {\rm For time-invariant systems, it is possible to
		improve the bounds of Proposition~\ref{prop:noise:radius:bounds}
		for $\fM_w$, $\fM_v$, and $\fR$ by exploiting the fact that the
		system and observer gain matrices are constant.  The precise
		bounds in this case (see also~\cite[Proposition
		5.5]{DB-JC-SM:20-acc}) are
		\begin{align*}
			\fM_w & \le\frac{1}{2}\sqrt{d}\rho_{\xi_0}
			+2\sqrt{q}\sum_{k=0}^{\ell_0-1}\|\Psi_kG\|\rho_w,
			\\
			\fM_v & \le 2M_vr\sum_{k=0}^{\ell_0-1}\|\Psi_kK\|, \quad \fR
			\le
			2\frac{C_v}{m_v}r^{\frac{p-1}{p}}\frac{\sum_{k=0}^{\ell_0-1}\|\Psi_kK\|}
			{\big(\sum_{k=0}^{\ell_0-1}\|\Psi_kK\|^p\big)^{\frac{1}{p}}},
		\end{align*}
		with $\ell_0$ as in the time-invariant case of
		Proposition~\ref{prop:noise:radius:bounds}, and where $G$ and $K$
		denote the constant values of the internal noise and observer gain
		matrices, resp.  The superiority of these bounds can be checked
		using the matrix bounds in
		Proposition~\ref{prop:fundamental:matrix:contraction}(i) and their
		derivation is based on a simplified version of the arguments
		employed for the proof of
		Proposition~\ref{prop:noise:radius:bounds}.} \oprocend
\end{remark}

\section{Application to economic dispatch with distributed energy resources}

In this section, we take advantage of the ambiguity sets constructed
with noisy partial measurements, cf.
Theorem~\ref{thm:ambiguity:set:noise}, to hedge against the
uncertainty in an optimal economic dispatch problem.  This is a
problem where uncertainty is naturally involved due to (dynamic)
energy resources, which the scheduler has no direct access to control
or measure, like storage or renewable energy elements. The financial
implications of the associated decisions are of utmost importance for
the electricity market and justify the use of a reliable decision
framework that accounts for the variability of the uncertain factors.

\subsection{Network model and optimization objective}

Consider a network with distributed energy
resources~\cite{AC-JC:18-tac} comprising of $n_1$ generator units and
$n_2$ storage (battery) units.
The network needs to operate as close as possible to a prescribed
power demand $D$ at the end of the time horizon $[0:\ell]$,
corresponding to a uniform discretization of step-size $\delta t$ of
the continuous-time domain.  To this end, each generator and storage
unit supplies the network with positive power $P^j$ and $S^\iota$,
respectively, at time $\ell$. We assume we can control the power of
the generators, which additionally needs to be within the upper and
lower thresholds $P_{\min}^j$ and $P_{\max}^j$, respectively.  Each
battery is modeled as an uncertain dynamic element with an unknown
initial state distribution and we can decide whether it is connected
($\eta^\iota=1$) or not ($\eta^\iota=0$) to the network at time
$\ell$.  Our goal is to minimize the energy cost while remaining as
close as possible to the prescribed power demand. Thus, we minimize
the overall cost
\begin{align}
	\mathcal C(\bm P,\bm \eta):=\sum_{j=1}^{n_1}g^j(P^j)
	+\sum_{\iota=1}^{n_2} \eta^\iota h^\iota(S^\iota)
	+c\bigg(\sum_{j=1}^{n_1}P^j+\sum_{\iota=1}^{n_2}\eta^\iota
	S^\iota-D\bigg)^2
	\label{dispatch:objective}
\end{align}
where $\bm P:=(P^1,\ldots,P^{n_1})$, $\bm
\eta:=(\eta^1,\ldots,\eta^{n_2})$, $g^j$ and $h^\iota$ are cost
functions for the power provided by generator $j$ and storage unit
$\iota$, respectively.  We treat the deviation of the injected power
from its prescribed demand as a soft constraint by assigning it a
quadratic cost with weight~$c$ and augmenting the overall cost
function~\eqref{dispatch:objective}. Due to the uncertainty about the
batteries' state and their injected powers $S^\iota$,
the minimization of~\eqref{dispatch:objective} is a stochastic
problem.

\subsection{Battery dynamics and observation model}

Each battery is modeled as a single-cell dynamic element and we
consider its current $I^\iota$ discharging over the operation interval
(if connected to the network) as a fixed and a priori known function
of time. Its dynamics is conveniently approximated by the equivalent
circuit in Figure~\ref{fig:equivalent:circuit}(a) (see e.g.,
\cite{ML:17,BYL-GN-RGJ-DHD:04}),
%
\begin{figure}
	\begin{center}
		\centering
		\subfigure[]{\includegraphics[width=.433\linewidth]{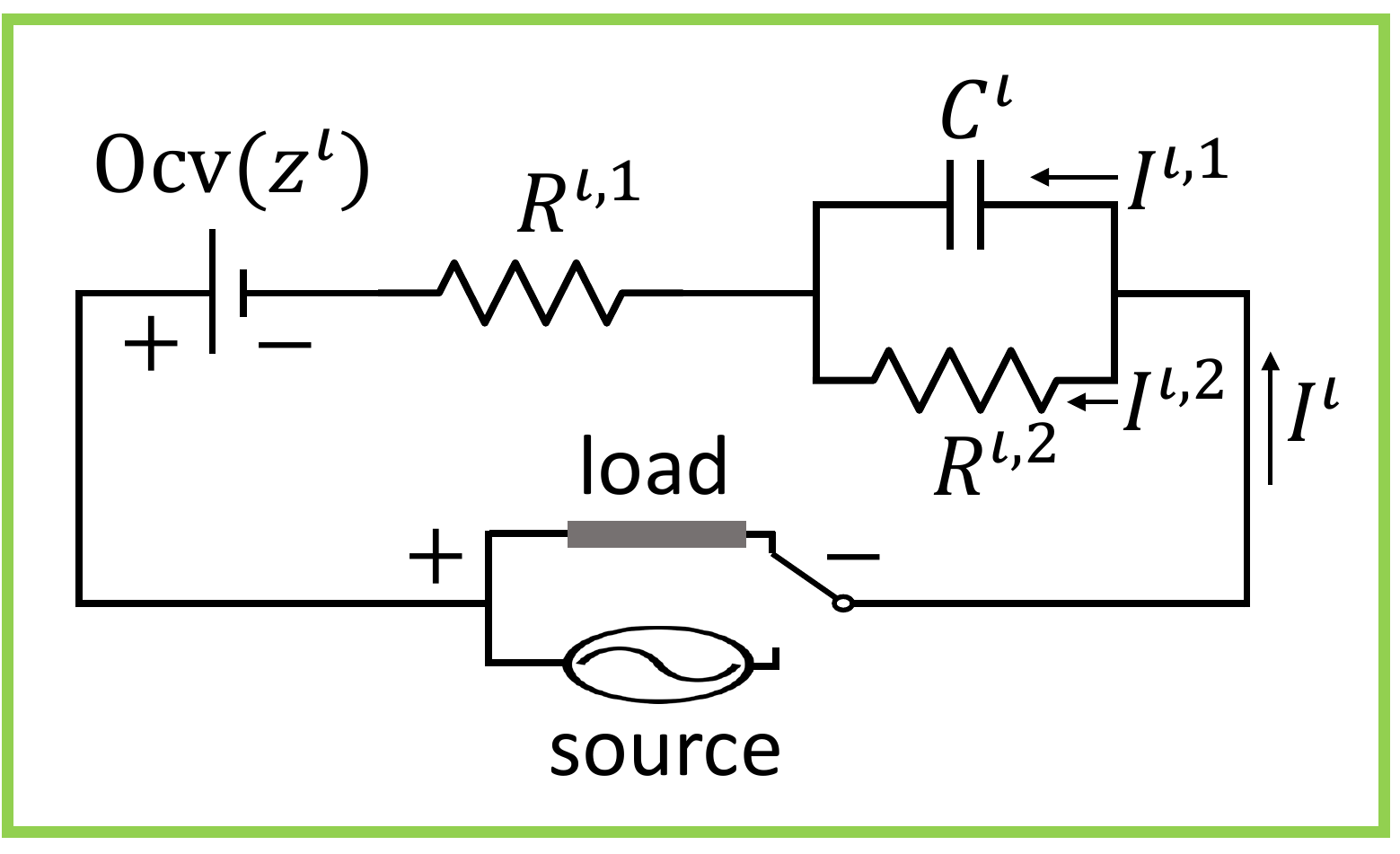}}
		%
		\hspace{2em}
		\subfigure[]{\includegraphics[width=.33\linewidth]{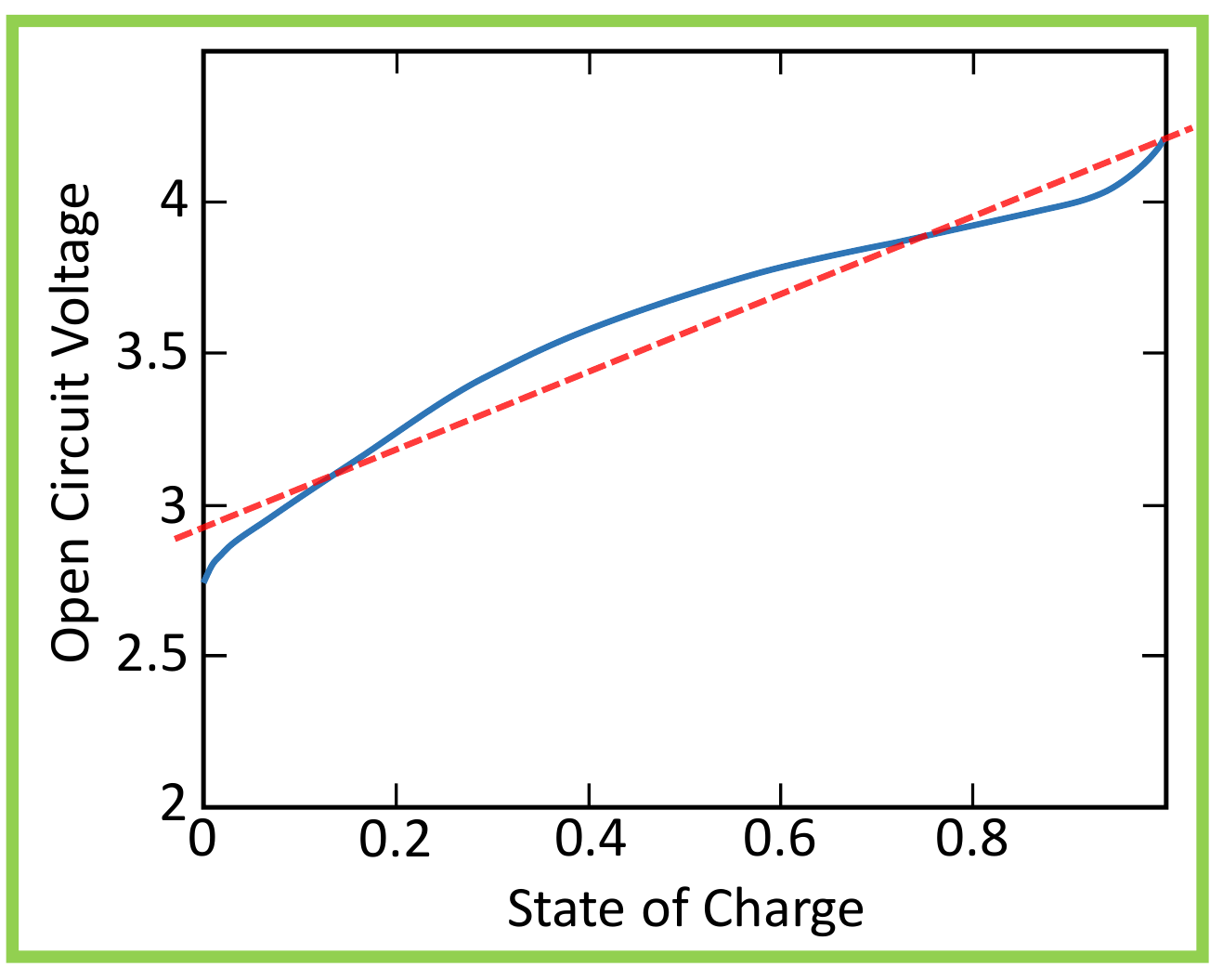}} 
		\caption{(a) shows the equivalent circuit model of a lithium-ion
			battery cell in discharging mode (c.f. \cite[Figure
			2]{BYL-GN-RGJ-DHD:04},\cite[Figure 1]{ML:17}). (b) is taken from
			\cite[Figure 3]{ML:17} and shows the nonlinear dependence of the
			open circuit voltage on the state of charge and its affine
			approximation.}\label{fig:equivalent:circuit}
	\end{center}
\end{figure}
where $z^\iota$ is the state of charge (SoC) of the cell and ${\rm
	Ocv}(z^\iota)$ is its corresponding open-circuit voltage, which we
approximate by the affine function $\alpha^\iota z^\iota+\beta^\iota$
in Figure~\ref{fig:equivalent:circuit}(b). The associated
discrete-time cell model is
\begin{align*}
	\chi_{k+1}^\iota\equiv \left(
	\begin{matrix}
		I_{k+1}^{\iota,2}
		\\
		z_{k+1}^\iota
	\end{matrix}\right)
	& = \left(
	\begin{matrix}
		a^\iota & 0
		\\
		0 & 1
	\end{matrix}
	\right)\left(
	\begin{matrix}
		I_k^{\iota,2} 
		\\
		z_k^\iota
	\end{matrix}
	\right) + \left(
	\begin{matrix}
		1-a^\iota
		\\
		-\delta t/Q^\iota
	\end{matrix}
	\right) I_k^\iota
	\\
	\theta_k^\iota\equiv V_k^\iota & = \alpha^\iota
	z_k^\iota+\beta^\iota
	-I_k^{\iota}R^{\iota,1}-I_k^{\iota,2}R^{\iota,2}
\end{align*}
where $a^\iota:=e^{-\delta t/(R^{2,\iota}C^\iota)}$, $\delta t$ is the
time discretization step, and $Q^\iota$ is the cell capacity.  Here,
we assume that for all $k\in[0:\ell]$ the cell is neither fully
charged or discharged (by e.g., requiring that
$0<z_0-\sum_{k=0}^{\ell-1}\delta tI_k^\iota/Q^\iota<1$ for all $k$ and
any candidate initial conditions and input currents) and so, the
evolution of its voltage is accurately represented by the above
difference equation. The initial condition comprising of the SoC
$z_0^\iota$ and the current $I_0^{\iota,2}$ through
$R^{\iota,2}$ is random with an unknown probability
distribution. We also consider additive measurement noise with an
unknown distribution, namely, we measure
\begin{align*}
	\theta_k^\iota = \alpha^\iota z_k^\iota+\beta^\iota
	-I_k^{\iota}R^{\iota,1}-I_k^{\iota,2}R^{\iota,2}+v_k.
\end{align*}   
To track the evolution of each random element through a linear system
of the form~\eqref{eq:data}, we consider for each battery a nominal
state trajectory $\chi_k^{\iota,\star}=(I_k^{\iota,2,\star},
z_k^{\iota,\star})$ initiated from the center of the support of its
initial-state distribution. Setting
$\xi_k^\iota=\chi_k^\iota-\chi_k^{\iota,\star}$ and
$\zeta_k^\iota=\theta_k(\chi_k^\iota)-\theta_k(\chi_k^{\iota,\star})$,
\begin{align*}
	\xi_{k+1}^\iota & =A_k^\iota\xi_k^\iota 
	\\
	\zeta_k^\iota & =H_k^\iota\xi_k^\iota+v_k,
\end{align*}
where $A_k^\iota:=\diag{a,1}$ and $H_k^\iota:=(\alpha^\iota,-R^{\iota,2})$.
Denoting $\bm\xi:=(\xi^1,\ldots,\xi^{n_2})$ and
$\bm\zeta:=(\zeta^1,\ldots,\zeta^{n_2})$, we obtain a system of the
form~\eqref{eq:data} for the dynamic random variable~$\bm\xi$. Despite
the fact that the state distribution $\bm \xi_k$ of the batteries
across time is unknown, we assume having access to output data from $N$
independent realizations of their dynamics over the horizon
$[0:\ell]$. Using these samples we exploit the results of the paper to
build an ambiguity ball $\P^N$ of radius $\eps_N$ in the 2-Wasserstein
distance (i.e., with $p=2$), that contains the batteries' state
distribution $P_{\bm \xi_\ell}$ at time $\ell$ with prescribed
probability $1-\beta$. In particular, we take the samples from each
realization $i\in[1:N]$ and use an observer to estimate its state
$\widehat{\bm\xi}_\ell^i$ at time $\ell$. The ambiguity set is
centered at the estimator-based empirical distribution $\widehat
P_{\bm\xi_\ell}^N=
\frac{1}{N}\sum_{i=1}^N\delta_{\widehat{\bm\xi}_\ell^i}$ and its
radius can be determined using Theorem~\ref{thm:ambiguity:set:noise}
and Proposition~\ref{prop:hatepsN:guarantees}.

\subsection{Decision problem as a distributionally robust
	optimization (DRO) problem}
To solve the decision problem regarding whether or not to connect the
batteries for economic dispatch, we formulate a distributionally
robust optimization problem for the cost~\eqref{dispatch:objective}
using the ambiguity set~$\P^N$. To do this, we derive an explicit
expression of how the cost function $\mathcal C$ depends on the
stochastic argument $\bm\xi_\ell$. Notice first that the power injected by 
each 
battery at time $\ell$ is
\begin{align*}
	S^\iota =I_\ell^\iota V_\ell^\iota & = I_\ell^\iota\big(\alpha^\iota
	z_\ell^\iota+\beta^\iota -I_\ell^{\iota}R^{\iota,1}-I_\ell^{\iota,2}
	R^{\iota,2}\big)
	\\
	& = \langle(-I_\ell^\iota R^{\iota,2},\alpha^\iota I_\ell^\iota),
	\chi_\ell^\iota \rangle+\beta^\iota
	I_\ell^\iota-(I_\ell^\iota)^2R^{\iota,1}=
	\langle\widehat\alpha^\iota,\xi_\ell^\iota\rangle+\widehat\beta^\iota
	\equiv(\widehat\alpha^\iota)^\top\xi_\ell^\iota+\widehat\beta^\iota,
\end{align*}
with $ \widehat\alpha^\iota :=(-I_\ell^\iota R^{\iota,2},\alpha^\iota
I_\ell^\iota)$ and
\begin{align*}
	\widehat\beta^\iota
	:=\langle\widehat\alpha^\iota,\chi_\ell^{\iota,\star}\rangle
	+I_\ell^\iota\beta^\iota-(I_\ell^\iota)^2R^{\iota,1} =I_\ell^\iota
	I_\ell^{\iota,2,\star}R^{\iota,2}-\alpha^\iota I_\ell^\iota
	z_\ell^{\iota,\star}+I_\ell^\iota\beta^\iota-(I_\ell^\iota)^2R^{\iota,1}.
\end{align*}
Considering further affine costs $h^\iota(S):=\bar \alpha^\iota
S+\bar\beta^\iota$ for the power provided by the batteries, the
overall cost $\mathcal C$ becomes
\begin{align}
	\mathcal C(\bm P,\bm\eta)=g(\bm P)
	+(\bm\eta*\bm{\widetilde\alpha})^\top\bm\xi_\ell
	+\bm\eta^\top\bm{\widetilde\beta} +c\big(\bone^\top\bm P
	+(\bm\eta*\bm{\widehat\alpha})^\top\bm\xi_\ell
	+\bm\eta^\top\bm{\widehat\beta}-D\big)^2, \label{dispatch:objectie:equiv}
\end{align}
where $*$ denotes the Khatri-Rao product (cf. Section~\ref{sec:prelims}) and
\begin{align*} 
	g(\bm P) & :=\sum_{j=1}^{n_1}g^j(P^j),\quad \bm{\widehat\alpha}
	:=(\widehat\alpha^1,\ldots,\widehat\alpha^{n_2}),\quad
	\bm{\widehat\beta}:=(\widehat\beta^1,\ldots,\widehat\beta^{n_2}), \\
	\bm{\widetilde\alpha} & :=(\bar\alpha^1\widehat\alpha^1,\ldots,
	\bar\alpha^{n_2}\widehat\alpha^{n_2}),\quad \bm{\widetilde\beta}
	:=(\bar\alpha^1\widehat\beta^1+\bar\beta^1,\ldots,
	\bar\alpha^{n_2}\widehat\beta^{n_2}+\bar\beta^{n_2}).
\end{align*}   
Using the equivalent description \eqref{dispatch:objectie:equiv} for
$\mathcal C$ and recalling the upper and lower bounds $P_{\min}^j$ and
$P_{\max}^j$ for the generator's power, we formulate the DRO power
dispatch problem
\begin{subequations}\label{dispacth:overall}
	\begin{align}
		\inf_{\bm \eta,\bm P}\; & \Big\{f_{\bm\eta}(\bm P)
		+\sup_{P_{\bm\xi_\ell}\in\P^N}\bE_{P_{\bm\xi_\ell}}
		\big[h_{\bm\eta}(\bm P,\bm\xi_\ell)\big]\Big\}, 
		\label{dispatch:objective:robust}  \\
		\textup{s.t.}\; & P_{\min}^j\le P^j \le
		P_{\max}^j\quad\forall j\in[1:n_1], \label{dispatch:constraint}
	\end{align}
\end{subequations}
with the ambiguity set $\P^N$ introduced above and
\begin{align*}
	f_{\bm\eta}(\bm P) & :=g(\bm P)+c\bm P^\top\bone\bone^\top\bm P
	+2c(\bm\eta^\top\bm{\widehat\beta}-D)\bone^\top\bm P +
	c(\bm\eta^\top\bm{\widehat\beta}-D)^2+\bm\eta^\top\bm{\widetilde\beta}
	\\
	h_{\bm \eta}(\bm P,\bm \xi_\ell) & :=
	c\bm\xi_\ell^\top(\bm\eta*\bm{\widehat\alpha})
	(\bm\eta*\bm{\widehat\alpha})^\top\bm\xi_\ell
	+\big(2c\big(\bone^\top\bm P+\bm\eta^\top\bm{\widehat\beta}-D)
	(\bm\eta*\bm{\widehat\alpha})^\top
	+(\bm\eta*\bm{\widetilde\alpha})^\top\big)\bm\xi_\ell,
\end{align*}
This formulation aims to minimize the worst-case expected cost with
respect to the plausible distributions of $\bm\xi$ at time~$\ell$.

\subsection{Tractable reformulation of the DRO problem}

Our next goal is to obtain a tractable reformulation of the
optimization problem~\eqref{dispacth:overall}. To this end, we first
provide an equivalent description for the inner maximization in
\eqref{dispacth:overall}, which is carried out over a space of
probability measures. Exploiting strong duality (see~\cite[Corollary
2(i)]{RG-AJK:16} or~\cite[Remark 1]{JB-KM:19}) and recalling that our
ambiguity set $\mathcal P^N$ is based on the 2-Wasserstein distance,
we equivalently write the inner maximization problem
$\sup_{P_{\bm\xi_\ell}\in\P^N} \bE_{P_{\bm\xi_\ell}}
\big[h_{\bm\eta}(\bm P,\bm\xi_\ell)\big]$ as
\begin{align}
	\inf_{\lambda \ge 0}\bigg\{\lambda \psi_N^2
	+\frac{1}{N}\sum_{i=1}^N\sup_{\xi_\ell\in\Xi}\{h_{\bm\eta}(\bm
	P,\bm\xi_\ell)
	-\lambda\|\bm\xi_\ell-\widehat{\bm\xi}_\ell^i\|^2\}\bigg\},
	\label{inner:maximization}
\end{align}
where $\psi_N\equiv\psi_N(\beta)$ is the radius of the
ambiguity ball, $\Xi\subset\Rat{2n_2}$ is the support of the
batteries' unknown state distribution, and the
$\widehat{\bm\xi}_\ell^i$ are the estimated states of their
realizations. We slightly relax the problem, by allowing the ambiguity
ball to contain all distributions with distance $\psi_N$ from
$\widehat P_{\bm\xi_\ell}^N$ that are supported on $\Rat{2n_2}$ and
not necessarily on $\Xi$. Thus, we first look to solve for each
estimated state $\widehat{\bm\xi}_\ell^i$ the optimization problem
\begin{align*}
	\sup_{\xi_\ell\in\Rat{2n_2}}\{h_{\bm\eta}(\bm P,\bm\xi_\ell)
	-\lambda\|\bm\xi_\ell-\widehat{\bm\xi}_\ell^i\|^2\},
\end{align*} 
which is written
\begin{align*}
	&
	\sup_{\bm\xi_\ell\in\Rat{2n_2}}\big\{\bm\xi_\ell^\top\fA\bm\xi_\ell
	+\big(2c\big(\bone^\top\bm P +\bm\eta^\top\bm{\widehat\beta}-D)
	(\bm\eta*\bm{\widehat\alpha})^\top
	+(\bm\eta*\bm{\widetilde\alpha})^\top\big)\bm\xi_\ell \\
	&\hspace{21.45em}
	-\lambda(\bm\xi_\ell-\widehat{\bm\xi}_\ell^i)^\top
	(\bm\xi_\ell-\widehat{\bm\xi}_\ell^i)\big\}
	\\
	& = -
	\lambda(\widehat{\bm\xi}_\ell^i)^\top\widehat{\bm\xi}_\ell^i+\sup_{\bm\xi_\ell\in\Rat{2n_2}}
	\big\{\bm\xi_\ell^\top(\fA-\lambda I_{2n_2})\bm\xi_\ell
	\\
	&\hspace{10.5em}+\big(2c\big(\bone^\top\bm
	P+\bm\eta^\top\bm{\widehat\beta}-D)
	(\bm\eta*\bm{\widehat\alpha})^\top
	+(\bm\eta*\bm{\widetilde\alpha})^\top
	+2\lambda(\widehat{\bm\xi}_\ell^i)^\top\big)\bm\xi_\ell\big\}
	\\
	& = - \lambda(\widehat{\bm\xi}_\ell^i)^\top\widehat{\bm\xi}_\ell^i
	+\sup_{\bm\xi_\ell\in\Rat{2n_2}}\big\{\bm\xi_\ell^\top(\fA-\lambda
	I_{2n_2}) \bm\xi_\ell+(\bm r^i)^\top\bm\xi_\ell\big\}
\end{align*}
where $\bm r^i\equiv \bm r_{\bm\eta}^i(\bm
P,\lambda):=2c(\bone^\top\bm P+\bm\eta^\top\bm{\widehat\beta}-D)
(\bm\eta*\bm{\widehat\alpha}) + \bm\eta*\bm{\widetilde\alpha}
+2\lambda\widehat{\bm\xi}_\ell^i$ and $\fA\equiv \fA_{\bm\eta}:= c
(\bm\eta*\bm{\widehat\alpha}) (\bm\eta*\bm{\widehat\alpha})^\top$ is a
symmetric positive semi-definite matrix with diagonalization
$\fA=\fQ^\top\fD\fQ$ where the eigenvalues decrease along the
diagonal. Hence, we get
\begin{align*}
	& \sup_{\bm\xi_\ell\in\Rat{2n_2}} \big\{\bm\xi_\ell^\top(\fA-\lambda
	I_{2n_2})\bm\xi_\ell+(\bm r^i)^\top\bm\xi_\ell\big\} \\
	& \hspace{10em} =\sup_{\bm\xi_\ell\in\Rat{2n_2}}
	\big\{\bm\xi_\ell^\top(\fQ^\top\fD\fQ-\fQ^\top\lambda
	I_{2n_2}\fQ)\bm\xi_\ell+(\bm r^i)^\top\bm\xi_\ell\big\}
	\\
	& \hspace{10em}=\sup_{\bm\xi\in\Rat{2n_2}}
	\big\{\bm\xi^\top(\fD-\lambda I_{2n_2})\bm\xi+(\bm{\widehat
		r}^i)^\top\bm\xi\big\}
\end{align*}
with $\bm{\widehat r}^i:=\fQ\bm r^i$ and denoting
$\lambda_{\max}(\fA)$ the maximum eigenvalue of $\fA$ we have
\begin{align}
	& \sup_{\bm\xi\in\Rat{2n_2}} \big\{\bm\xi^\top(\fD-\lambda
	I_{2n_2})\bm\xi+(\bm{\widehat r}^i)^\top\bm\xi\big\}=
	\begin{cases}
		\infty & {\rm if}\;0\le\lambda<\lambda_{\max}(\fA) \\
		\frac{1}{4}(\bm{\widehat r}^i)^\top(\lambda
		I_{2n_2}-\fD)^{-1}\bm{\widehat r}^i & {\rm
			if}\;\lambda>\lambda_{\max}(\fA).
	\end{cases} \label{inner:sup:removal}
\end{align}
To obtain this we exploited that $Q(\bm \xi):=\bm\xi^\top(\fD-\lambda 
I_{2n_2})\bm\xi+(\bm{\widehat r}^i)^\top\bm\xi$ is maximized when 
\begin{align*}
	\nabla Q(\bm \xi_\star)=0\iff 2(\fD-\lambda
	I_{2n_2})\bm\xi_\star+\bm{\widehat
		r}^i=0\iff\bm\xi_\star=\frac{1}{2}(\lambda
	I_{2n_2}-\fD)^{-1}\bm{\widehat r}^i,
\end{align*}
which gives the optimal value
$Q(\bm\xi_\star)=\frac{1}{4}(\bm{\widehat r}^i)^\top(\lambda
I_{2n_2}-\fD)^{-1}\bm{\widehat r}^i$. Note that we do not need to
specify the value of the expression in \eqref{inner:sup:removal} for
$\lambda=\lambda_{\max}$. In particular, since the function we
minimize in \eqref{inner:maximization} is convex in $\lambda$, the
inner part of the DRO problem is equivalently written
\begin{align*}
	\inf_{\lambda>\lambda_{\max}(\fA)}\bigg\{\lambda\bigg( \psi_N^2-
	\frac{1}{N}\sum_{i=1}^N(\widehat{\bm\xi}_\ell^i)^\top\widehat{\bm\xi}_\ell^i\bigg)
	+\frac{1}{4N}\sum_{i=1}^N \bm{\widehat r}_{\bm\eta}^i(\bm
	P,\lambda)^\top(\lambda I_{2n_2}-\fD)^{-1} \bm{\widehat
		r}_{\bm\eta}^i(\bm P,\lambda) \bigg\}.
\end{align*}
Taking further into account that
\begin{align*}
	(\lambda I_{2n_2}-\fD)^{-1}={\rm diag}\Big(\frac{1}{\lambda
		-\lambda_{\max}(\fA)},\ldots,\frac{1}{\lambda-\lambda_{\min}(\fA)}\Big),
\end{align*}
as well as the constraints \eqref{dispatch:constraint} on the decision
variable $\bm P$, the overall DRO problem is reformulated as
\begin{subequations} \label{DRO:reformulated}
	\begin{align}
		\min_{\bm\eta}\inf_{\bm P,\lambda} & \bigg\{f_{\bm\eta}(\bm
		P)+\lambda\bigg(\psi_N^2-
		\frac{1}{N}\sum_{i=1}^N(\widehat{\bm\xi}_\ell^i)^\top\widehat{\bm\xi}_\ell^i\bigg)
		+\frac{1}{4N}\sum_{i=1}^N \bm{\widehat r}_{\bm\eta}^i(\bm
		P,\lambda)^\top \\
		& \hspace{3.5em}\times{\rm
			diag}\Big(\frac{1}{\lambda-\lambda_{\max}(\fA)},
		\ldots,\frac{1}{\lambda -\lambda_{\min}(\fA)}\Big) \bm{\widehat
			r}_{\bm\eta}^i(\bm P,\lambda) \bigg\} \nonumber \\
		\textup{subject to}\; &
		P_{\min}^j\le P^j \le P_{\max}^j\quad\forall j\in[1:n_1]\\
		& \lambda>\lambda_{\max}(\fA). \nonumber
	\end{align}
\end{subequations}

\subsection{Simulation results}

For the simulations we consider $n_1=4$ generators and $n_2=3$
batteries with the same characteristics. We assume that the
distributions of each initial SoC $z_0^\iota$ and current
$I_0^{\iota,2}$ are known to be supported on the intervals $[0.45,
0.9]$ and $[1.5,1.7]$, respectively. The true SoC distribution for
batteries 2 and 3 at time zero is $P_{z_0^2}=P_{z_0^3}=\mathcal
U[0.45,0.65]$ ($\mathcal U$ denotes uniform distribution). On the
other hand, the provider of battery 1 has access to the distinct
batteries 1A and 1B and selects randomly one among them with
probabilities 0.9 and 0.1, respectively. The SoC distribution of
battery 1A at time zero is $P_{z_0^{1A}}=\mathcal U[0.45,0.65]$,
whereas that of battery 1B is $P_{z_0^{1B}}=\mathcal U[0.84,0.86]$.
Thus, we get the bimodal distribution $P_{z_0^1}=0.9\mathcal
U[0.45,0.65]+0.1\mathcal U[0.84,0.86]$, which is responsible for
non-negligible empirical distribution variations, since for small
numbers of samples, it can fairly frequently occur that the relative
percentage of samples from 1B deviates significantly from its expected
one. On the other hand, we assume that the true initial currents
$I_0^{\iota,2}$ of all batteries are fixed to 1.6308, namely,
$P_{I_0^{1,2}}=P_{I_0^{2,2}}=P_{I_0^{3,2}}=\delta_{1.6308}$.
For the measurements, we consider the Gaussian mixture noise model 
$P_{v_k}=0.5\mathcal N(0.01,0.01^2)+0.5\mathcal N(-0.01,0.01^2)$ with 
$\mathcal N(\mu,\sigma^2)$ denoting the normal distribution with mean $\mu$ 
and 
variance $\sigma^2$.  

\begin{figure}[]
	\begin{center}
		\centering
		\includegraphics[width=.85\linewidth]{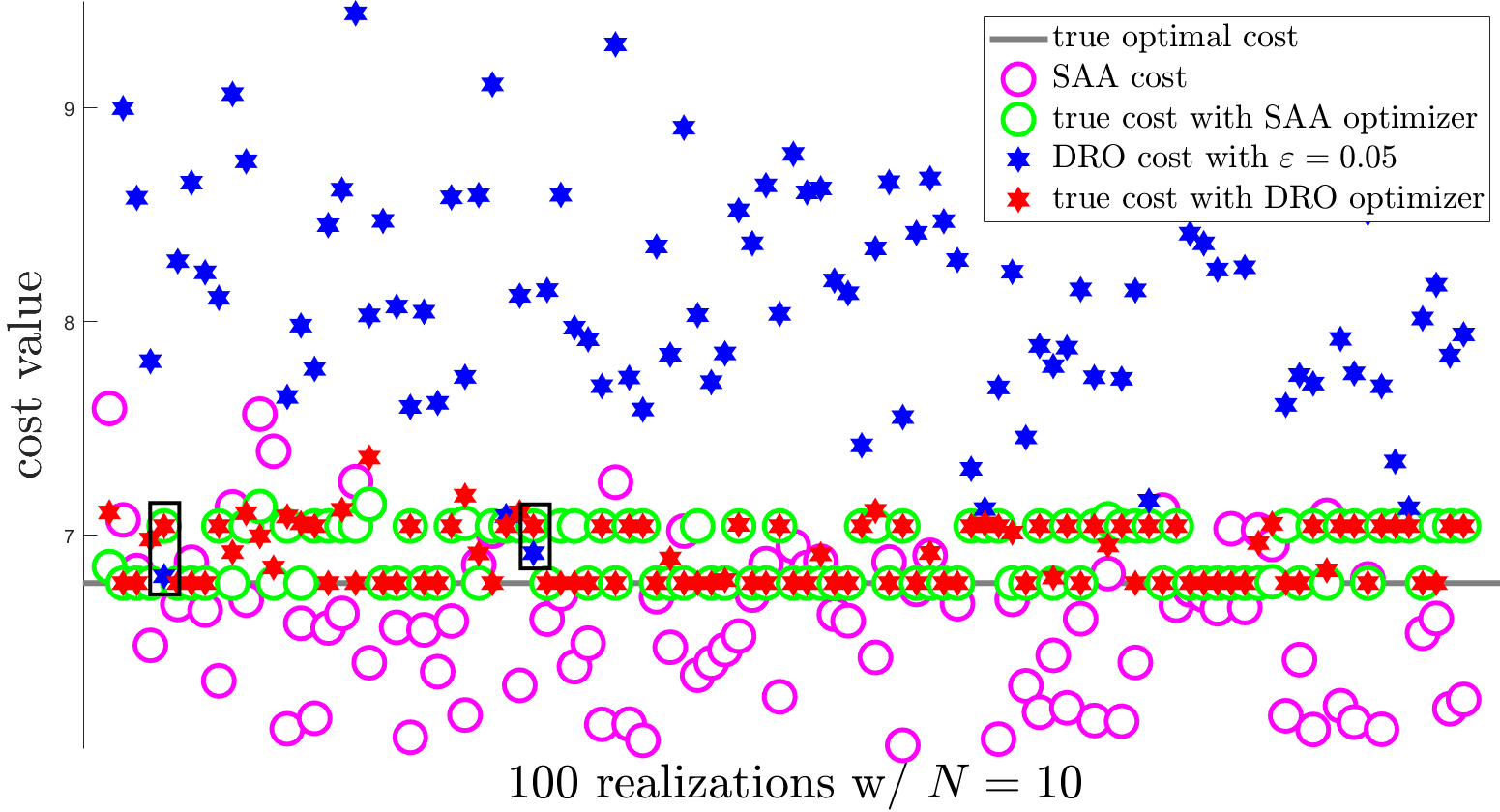}
		\caption{Results from 100 realizations of the power dispatch
			problem with $N=10$ independent samples used for each
			realization. We compute the optimizers of the SAA and DRO
			problems, plot their corresponding optimal values (termed ``SAA
			cost'' and ``DRO cost''), and also evaluate their 
			performance with respect to the true distribution 
			(``true cost with SAA optimizer'' and ``true cost with DRO
			optimizer'').  With the exception of two realizations (whose DRO
			cost and true cost with the DRO optimizer are framed inside black 
			boxes), the DRO cost is above the true cost of the DRO 
			optimizer, namely, this happens with high probability. From the plot, 
			it is also clear that the SAA solution tends to over-promise since its
			value is most frequently below the true cost of the SAA
			optimizer.}\label{fig:SAAvsDRO:Neq10}
	\end{center}
\end{figure}
%
%
To compute the ambiguity radius for the reformulated DRO problem
\eqref{DRO:reformulated}, we specify its nominal and noise components
$\eps_N(\subscr{\beta}{nom},\rho_{\bm\xi_\ell})$ and
$\widehat\eps_N(\subscr{\beta}{ns})$, where due to
Proposition~\ref{prop:nominal:radius}, $\rho_{\bm\xi_\ell}$ can be
selected as half the diameter of any set containing the support of
$P_{\bm\xi_\ell}$ in the infinity norm. It follows directly from the
specific dynamics of the batteries that $\rho_{\bm\xi_\ell}$ does not
exceed half the diameter of the initial conditions' distribution
support, which is isometric to $[0.45, 0.9]^3\times
[1.5,1.7]^3\subset\Rat{6}$. Hence, using 
Proposition~\ref{prop:explicit:ambiguity:radius} with $p=2$, $d=6$, and
$\rho_{\bm\xi_\ell}=0.225$, we obtain
\begin{align*}
	\eps_N(\subscr{\beta}{nom},\rho_{\bm\xi_\ell})
	%
	=4.02N^{-\frac{1}{6}}+1.31(\ln\subscr{\beta}{nom}^{-1})^{\frac{1}{4}}N^{-\frac{1}{4}}
	.
\end{align*}
To determine the noise radius, we first compute lower and upper bounds
$m_v$ and $M_v$ for the $L_2$ norm of the Gaussian mixture noise $v_k$
and an upper bound $C_v$ for its $\psi_2$ norm. Denoting by $\bE_P$
the integral with respect to the distribution $P$, we have for
$P_{v_k}=0.5\mathcal N(\mu_1,\sigma_1^2)+0.5\mathcal
N(\mu_2,\sigma_2^2)$ that
$\|v_k\|_2^2=\bE_{\frac{1}{2}(P_1+P_2)}\big[v_k^2\big]=\frac{1}{2}(\mu_1^2
+\sigma_1^2+\mu_2^2+\sigma_2^2)$, where $P_1=\mathcal
N(\mu_1,\sigma_1^2)$, $P_2=\mathcal N(\mu_2,\sigma_2^2)$ and we used
the fact that
$\bE_{P_i}\big[v_k^2\big]=\mu_i^2+\bE_{P_i}\big[(v_k-\mu_i)^2\big]
=\mu_i^2+\sigma_i^2$. Hence, in our case, where $\mu_i=\sigma_i=0.01$,
we can pick $m_v=M_v=0.01\sqrt{2}$. Further, using 
Proposition~\ref{prop:Gaussian:mixture:psi2:norm}, we can select
$C_v=0.01(\sqrt{8/3}+\sqrt{\ln 2})$.
To perform the state estimation from the output samples we used a Kalman 
filter. Its initial condition covariance matrix corresponds to independent 
Gaussian  distributions for each SoC $z_0^\iota$ and current 
$I_0^{\iota,2}$ with a standard deviation of the order of their assumed 
support. 
%
%
We also select the same covariance as in the components of the
Gaussian mixture noise to model the measurement noise of the Kalman
filter. Using the dynamics of the filter and the values of $m_v$,
$M_v$, and $C_v$ above, we obtain from \eqref{frakMw}, 
\eqref{frakMv}-\eqref{frakmv}, and \eqref{constant:frakR} the constants 
$\fM_w=0.325$, $\fM_v=0.008$, and $\fR=2.72$ for the expression of the noise
radius. In particular, we have from
Proposition~\ref{prop:hatepsN:guarantees} that
$\widehat\eps_N(\subscr{\beta}{ns})=0.47
+0.0113\sqrt{74.98/N\ln(2/\subscr{\beta}{ns})}$ and the overall
radius~is
\begin{align}
	\psi_N(\beta)=0.47+4.02N^{-\frac{1}{6}} 
	+1.31(\ln\subscr{\beta}{nom}^{-1})^{\frac{1}{4}}N^{-\frac{1}{4}}
	+0.0973(\ln(2\subscr{\beta}{ns}^{-1}))^{\frac{1}{2}}N^{-\frac{1}{2}}. 
	\label{psiN:example}
\end{align}

\begin{figure}[]
	\begin{center}
		\centering
		\includegraphics[width=.85\linewidth]{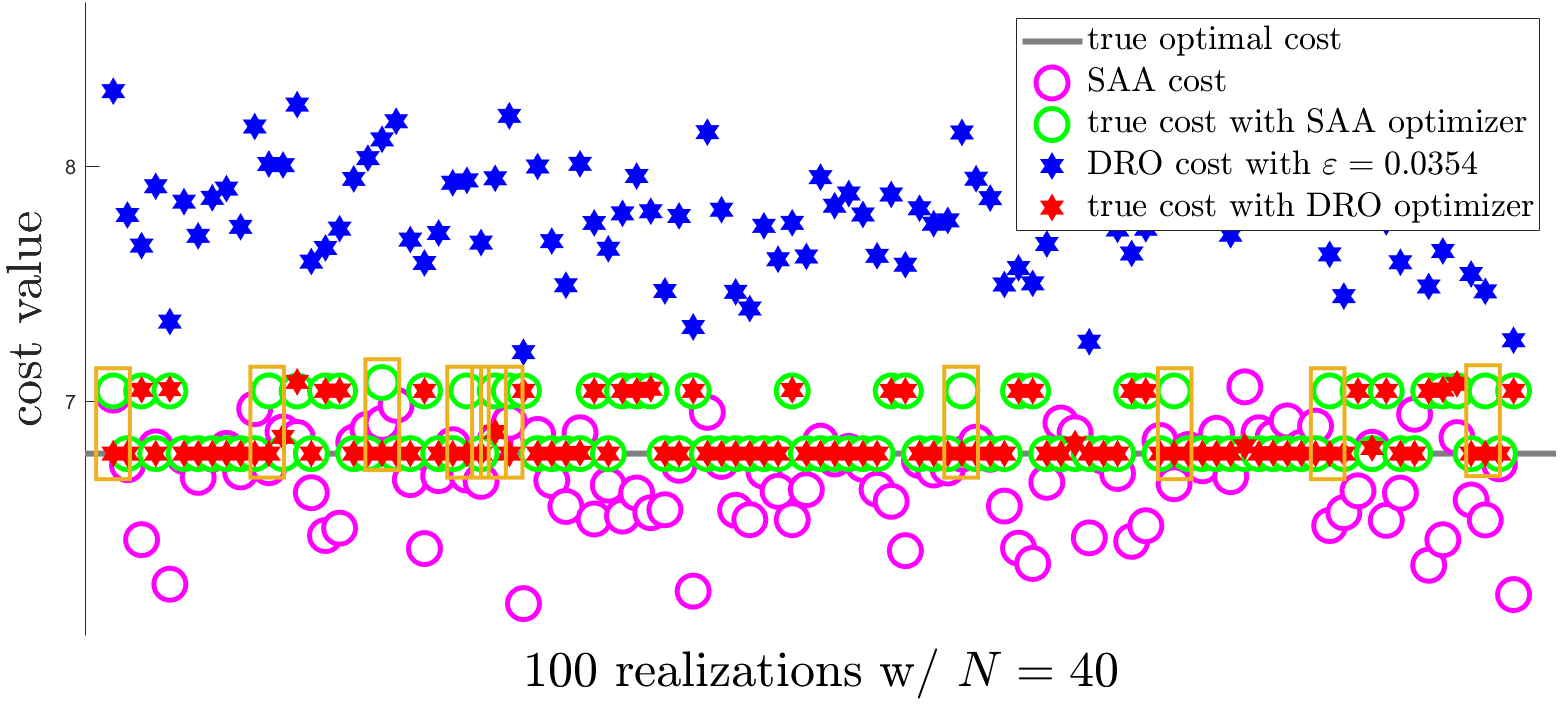}
		\\ (a) \\
		\includegraphics[width=.85\linewidth]{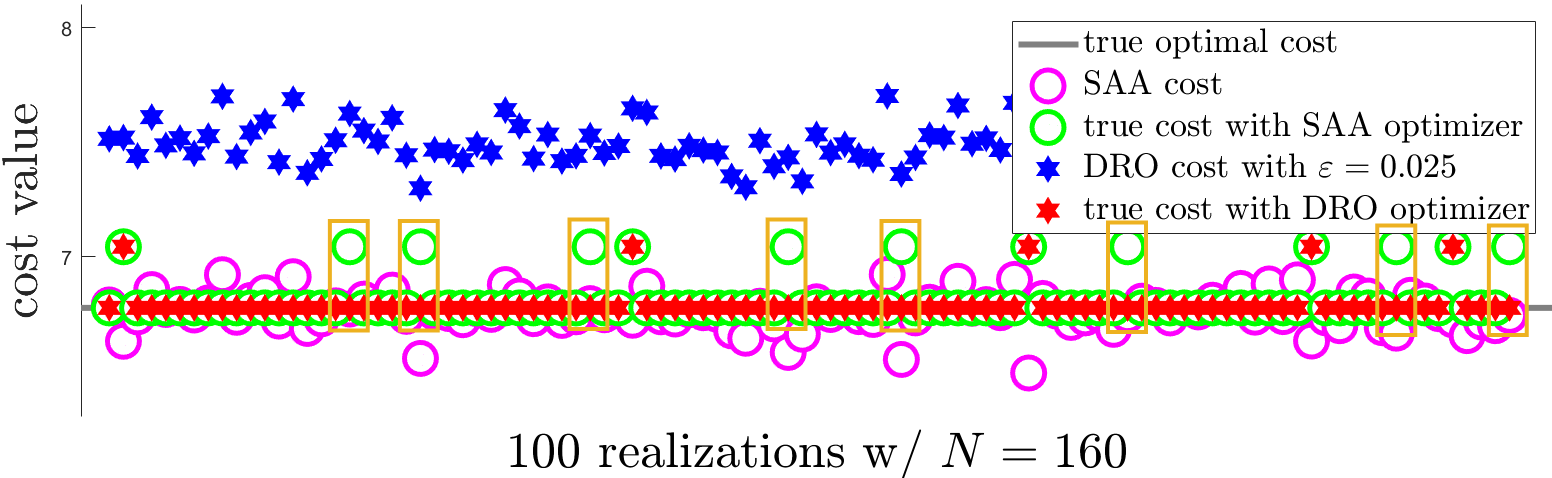}
		\\ (b)
		\caption{Analogous results to those of
			Figure~\ref{fig:SAAvsDRO:Neq10}, from 100 realizations with (a)
			$N=40$ and (b) $N=160$ independent samples, and the ambiguity
			radius tuned so that the same confidence level is preserved.  In
			both cases, the DRO cost is above the true cost of the DRO
			optimizer with high probability (in fact, always). Furthermore,
			the cost of the DRO optimizer (red star) is strictly
			better than the true cost of the SAA one (green circle) for
			a considerable number of realizations (highlighted in the
			illustrated boxes).}
		\label{fig:SAAvsDRO:Neq40}
	\end{center}
\end{figure}


We assume that the energy cost of the generators is lower than that of
the batteries and select the quadratic power generation cost $g(\bm
P)=0.25\sum_{j=1}^4(P^j-0.1)^2$ and the same lower/upper power
thresholds $P_{\min}^j=0.2$/$P_{\max}^j=0.5$ for all generators.  For
the batteries, we pick the same resistances $R^{\iota,1}=0.34$ and
$R^{\iota,2}=0.17$, and we take $a^\iota=0.945$ and $I_k^\iota=8$ for
all times. We nevertheless use different linear costs
$h^\iota(S)=\bar\alpha^\iota S$ for their injected powers, with
$\bar\alpha^1=1$ and $\bar\alpha^2=\bar\alpha^3=1.3$, since battery~1
is less reliable due to the large SoC fluctuation among its two modes.

We solve 100 independent realizations of the overall economic dispatch
problem. For each of them, we generate independent samples from the
batteries' initial condition distributions and solve the associated
sample average approximation (SAA) and DRO problems for $N=10$,
$N=40$, and $N=160$ samples, respectively, using
CVX~\cite{website:cvx}.  It is worth noting that the radius $\psi_N$
given by~\eqref{psiN:example} is rather conservative. The main reasons
for this are 1) conservativeness of the concentration of measure
results used for the derivation of the nominal radius, 2) lack of
homogeneity of the distribution's support (the a priori known support
of the $I_0^{\iota,2}$ components is much smaller than that of the
$z_0^{\iota}$ ones), 3) independence of the batteries' individual
distributions, which we have not exploited, and 4) conservative upper
bounds for the estimation error. Although there is room to sharpen all
these aspects, it requires multiple additional contributions and lies
beyond the scope of the paper. Nevertheless, the
formula~\eqref{psiN:example} gives a qualitative intuition about the
decay rates for the ambiguity radius. In particular, it indicates that
under the same confidence level and for small sample sizes, an
ambiguity radius proportional to $N^{-\frac{1}{4}}$ is a reasonable
choice. Based on this, we selected the ambiguity radii $0.05$,
$0.0354$, and $0.025$ for $N=10$, $N=40$, and $N=160$. The associated
simulation results are shown in
Figures~\ref{fig:SAAvsDRO:Neq10},~\ref{fig:SAAvsDRO:Neq40}(a),
and~\ref{fig:SAAvsDRO:Neq40}(b), respectively.  We plot there the
optimal values of the SAA and DRO problems (termed ``SAA cost'' and
``DRO cost'') and provide the expected performance of their respective
decisions with respect to the true distribution (``true cost with
SAA optimizer'' and ``true cost with DRO optimizer'').
We observe that in all three cases, the DRO value is above the
true cost of the DRO optimizer for nearly all realizations (and
for all when $N$ is 40 or 160), which verifies the out-of-sample
guarantees that we seek in DRO formulations \cite[Theorem
3.5]{PME-DK:17}. In addition, when solving the problem for 40 or 160
samples, we witness a clear superiority of the DRO decision compared
to the one of the non-robust SAA, because it considerably improves the
true cost for a significant number of realizations
(cf. Figure~\ref{fig:SAAvsDRO:Neq40}).

\subsection{Discussion}

The SAA solution tends to consistently promise a better outcome
compared to what the true distribution reveals for the same decision
(e.g., magenta circle being usually under the green circle in all
figures). This rarely happens for the DRO solution, and when it does,
it is only by a small margin.  This makes the DRO approach preferable
over the SAA one in the context of power systems operations where
honoring committments at a much higher cost than anticipated might
result in significant losses, and not fulfilling committments may lead
to penalties from the system operator.

\section{Conclusions} \label{sec:conclusions}

We have constructed high-confidence ambiguity sets for dynamic random
variables using partial-state measurements from independent
realizations of their evolution. In our model, both the dynamics and
measurements are subject to disturbances with unknown probability
distributions. The ambiguity sets are built using an observer to
estimate the full state of each realization and leveraging
concentration of measure inequalities. For systems that are either
time-invariant and detectable, or uniformly observable, we have
established uniform boundedness of the ambiguity radius. To aid the
associated probabilistic guarantees, we also provided auxiliary
concentration of measure results. Future research will include the
consideration of robust state estimation criteria to mitigate the
noise effect on the ambiguity radius, the extension of the results to
nonlinear dynamics, and the construction of ambiguity sets with
information about the moments.

\section{Appendix}




Here we give proofs of various results of the paper.
and provide explicit constants for the ambiguity radius.

\subsection{Technical proofs}

\begin{proof}[Proof of Lemma~\ref{lemma:expectation:inequality}]
	By independence of $X$ and $Y$, their joint distribution
	$P_{(X,Y)}$ is the product measure $P_X\otimes P_Y$ of the
	individual distributions $P_X$ and $P_Y$.
	Thus, from the Fubini theorem~\cite[Theorem 2.6.5]{RBA:72} and
	integrability of $g$, $k$ we get
	\begin{align*}
		\bE[g(X,Y)] & =\int_{\Rat{n_1}\times\Rat{n_2}}g(x,y)dP_{(X,Y)} 
		=\int_{\Rat{n_1}}\int_{\Rat{n_2}}g(x,y)dP_YdP_X \\
		& =\int_{\Rat{n_1}}\bE[g(x,Y)]dP_X=\int_K\bE[g(x,Y)]dP_X \\
		& \ge\int_Kk(x)dP_X=\int_{\Rat{n_1}}k(x)dP_X=\bE[k(X)],
	\end{align*} 
	which concludes the proof.
\end{proof}

\begin{proof}[Proof of Lemma~\ref{lemma:empirical:vs:true}]
	Using \cite[Lemma A.2]{DB-JC-SM:21-tac} to bound the Wasserstein
	distance of two discrete distributions, we get 
	\begin{align*}
	W_p(\widehat P_{\xi_\ell}^N,P_{\xi_\ell}^N) \le\Big(\frac{1}{N}\sum_{i=1}^N\|
	\widehat\xi_{\ell}^i-\xi_{\ell}^i\|^p\Big)^\frac{1}{p} =
	\Big(\frac{1}{N}\sum_{i=1}^N\|e_{\ell}^i\|^p\Big)^\frac{1}{p}.
	\end{align*}
	From~\eqref{error:formula}, we have
	\begin{align*}
		\|e_{\ell}^i\|& =\Big\|\Psi_\ell
		z^i+\sum_{k=1}^\ell\big(\Psi_{\ell,\ell-k+1}
		G_{\ell-k}\omega_{\ell-k}^i+\Psi_{\ell,\ell-k+1}K_{\ell-k}v_{\ell-k}^i\big)
		\Big\|
		\\
		& \le
		\|\Psi_\ell\|\|z^i\|+\sum_{k=1}^\ell\|\Psi_{\ell,\ell-k+1}G_{\ell-k}\|
		{\|\omega_{\ell-k}^i\|}
		+\sum_{k=1}^\ell\|\Psi_{\ell,\ell-k+1}K_{\ell-k}\| {\|v_{\ell-k}^i\|_1}
		\\
		& =: {\fM(z^i,\bm\omega^i)+\fE(\bm v^i),}
	\end{align*}
	with $\fE(\bm v^i)\equiv\fE^i$ given in the statement. Since
	$(a+b)^p\le 2^{p-1}(a^p+b^p)$ for $a,b\ge 0$ and $p\ge 1$,
	\begin{align*}
		W_p(\widehat P_{\xi_\ell}^N,P_{\xi_\ell}^N)\le
		\Big(\frac{1}{N}2^{p-1}
		\sum_{i=1}^N({\fM(z^i,\bm\omega^i)^p+(\fE^i)^p})\Big)^\frac{1}{p}.
	\end{align*}
	Next, using $(a+b)^{\frac{1}{p}}\le a^{\frac{1}{p}}+b^{\frac{1}{p}}$ for 
	$a,b\ge 0$ and $p\ge 1$, we have
	\begin{align}
		W_p(\widehat P_{\xi_\ell}^N,P_{\xi_\ell}^N)\le &
		\Big(\frac{1}{N}2^{p-1}\sum_{i=1}^N{\fM(z^i,\bm\omega^i)^p}\Big)^\frac{1}{p}
		+\Big(\frac{1}{N}2^{p-1}\sum_{i=1}^N{(\fE^i)^p}\Big)^\frac{1}{p}.
		\label{empirical:true:vs:estimated:bound2}
	\end{align}
	Finally, since $(\bm z,\bm \omega)\in
	B_\infty^{Nd}(\rho_{\xi_0})\times B_\infty^{N\ell q}(\rho_w)$, we
	get 
	\begin{align*}
	\fM(z^i,\bm\omega^i)^p\le\|\Psi_\ell\|\sqrt{d}\|z^i\|_{\infty}
	+\sum_{k=1}^\ell\|\Psi_{\ell,\ell-k+1}G_{\ell-k}\|
	\sqrt{q}\|\omega_{\ell-k-1}^i\|_{\infty} \le \fM_w.
	\end{align*}
  	This combined with~\eqref{empirical:true:vs:estimated:bound2}
	yields~\eqref{empirical:true:vs:estimated:bound1}.
\end{proof}

\begin{proof}[Proof of Lemma~\ref{lemma:Orlicz:bounds}]
	From \textbf{H4} in Assumption~\ref{assumption:distributiuon:class},
	we obtain for each summand in~\eqref{frak:Eis} 
	\begin{align*}
		\big\|\|\Psi_{\ell,\ell-k+1}K_{\ell-k}\|\|v_{\ell-k}^i\|_1\big\|_{\psi_p}
		&
		\le\|\Psi_{\ell,\ell-k+1}K_{\ell-k}\|\big(\big\|v_{\ell-k,1}^i\big\|_{\psi_p}
		+\cdots+\big\|v_{\ell-k,r}^i\big\|_{\psi_p}\big) \\
		& \le C_vr\|\Psi_{\ell,\ell-k+1}K_{\ell-k}\|.
	\end{align*}
	%
	Hence, we deduce that 
	\begin{align*}
		\|\fE^i\|_{\psi_p} \le
		\sum_{k=1}^\ell\big\|\|\Psi_{\ell,\ell-k+1}K_{\ell-k}\|
		\|v_{\ell-k}^i\|_1\big\|_{\psi_p} \le
		C_vr\sum_{k=1}^\ell\|\Psi_{\ell,\ell-k+1}K_{\ell-k}\|.
	\end{align*}
	For the $L^p$ bounds, note that $ \|\fE^i\|_p=
	\big\|\sum_{k\in[1:\ell],l\in[1:r]}
	\|\Psi_{\ell,\ell-k+1}K_{\ell-k}\||v_{\ell-k,l}^i|\big\|_p$.  Thus,
	from the inequality $\|\sum_i c_iX_i\|_p\le \sum_i c_i\|X_i\|_p$,
	which holds for any nonnegative $c_i$ and $X_i$ in $L^p$,
	\begin{align*}
		\|\fE^i\|_p\le \sum_{k\in[1:\ell],l\in[1:r]}
		\|\Psi_{\ell,\ell-k+1}K_{\ell-k}\|\|v_{\ell-k,l}^i\|_p,
	\end{align*}
	which, by the upper bound in \textbf{H4} of
	Assumption~\ref{assumption:distributiuon:class},
	implies~\eqref{frakMv}. For the other bound, we exploit linearity of
	the expectation and the inequality $\big(\sum_i c_i\big)^p\ge \sum_i
	c_i^p$, which holds for any nonnegative $c_i$, to get
	\begin{align*}
		\big(\bE\big[(\fE^i)^p\big]\big)^{\frac{1}{p}} \ge
		\bigg(\sum_{k\in[1:\ell],l\in[1:r]}\|\Psi_{\ell,\ell-k+1}K_{\ell-k}\|^p
		\bE\big[|v_{\ell-k,l}^i|\big]^p\bigg)^{\frac{1}{p}}.
	\end{align*} 
	Thus, from the lower bound in \textbf{H4} of
	Assumption~\ref{assumption:distributiuon:class} we also
	obtain~\eqref{frakmv}.
\end{proof}

We next prove Proposition~\ref{prop:concentration:around:mean}, along
the lines of the proof of~\cite[Theorem 3.1.1]{RV:18}, which considers
the special case of sub-Gaussian distributions. We rely on the
following concentration inequality~\cite[Corollary 2.8.3]{RV:18}.

\begin{proposition}\longthmtitle{Bernstein
		inequality}\label{prop:Bernstein:inequality} 
	Let $X_1,\ldots, X_N$ be scalar, mean-zero, sub-exponential,
	independent random variables. Then, for every $t\ge 0$ we have
	\begin{align*}
		\bP\bigg(\Big|\frac{1}{N}\sum_{i=1}^NX_i\Big|\ge t\bigg)
		\le 2\exp \Big(-c'\min\Big\{\frac{t^2}{R^2},\frac{t}{R}\Big\}N\Big),
	\end{align*}   
	where $c'=1/10$ and $R:=\max_{i\in[1:N]}\|X_i\|_{\psi_1}$.
\end{proposition}

The precise constant $c'$ above is not specified in \cite{RV:18} but
we provide an independent proof of this result in 
Section~\ref{subsec:explicit:contants}.

\begin{proof}[Proof of
	Proposition~\ref{prop:concentration:around:mean}]
	Note that each random variable $X_i^p-1$ is mean zero by
	assumption.  Additionally, we have that $\|X_i^p-1\|_{\psi_1} \le
	\|X_i^p\|_{\psi_1}+\|1\|_{\psi_1}=\|X_i\|_{\psi_p}+1/\ln 2\le R$,
	where we took into account that
	\begin{align*}
		\bE[\psi_1(X_i^p/t^p)]=\bE[\psi_p(X_i/t)]\Rightarrow
		\|X_i^p\|_{\psi_1}=\|X_i\|_{\psi_p},
	\end{align*}
	and the following fact, shown after the proof.
	
	\noindent $\triangleright$ \textit{Fact I.} For any constant random
	variable $X=\mu\in\Rat{}$, it holds $\|X\|_{\psi_p}=|\mu|/(\ln
	2)^{\frac{1}{p}}$.  \quad $\triangleleft$
	
	Thus, we get from Proposition~\ref{prop:Bernstein:inequality} that 
	\begin{align}
		\bP\bigg(\bigg|\frac{1}{N}\sum_{i=1}^NX_i^p-1\bigg|\ge t\bigg) \le
		2\exp \Big(-\frac{c'N}{R^2}\min\{t^2,t\}\Big),
		\label{pth:power:mean:concentr} 
	\end{align} 
	where we used the fact that $R>1$. 
	We will further leverage the following facts shown after the proof of the 
	proposition.
	
	\noindent $\triangleright$ \textit{Fact II.} For all $p\ge 1$ and $z\ge 0$ 
	it 
	holds that 
	$
	|z-1|\ge\delta\Rightarrow|z^p-1|\ge \max\{\delta,\delta^p\}. \quad 
	\triangleleft
	$ 
	
	\noindent $\triangleright$ \textit{Fact III.} For any $\delta\ge0$, if 
	$u=\max\{\delta,\delta^p\}$, then  $\min\{u,u^2\}=\alpha_p(\delta)$, with 
	$\alpha_p$ as given by~\eqref{function:alphap}. \quad $\triangleleft$
	
	By exploiting Fact II, we get
	\begin{align*}
		&
		\bP\bigg(\bigg|\bigg(\frac{1}{N}\sum_{i=1}^NX_i^p\bigg)^{\frac{1}{p}}
		-1\bigg|\ge t\bigg) \le
		\bP\bigg(\bigg|\frac{1}{N}\sum_{i=1}^NX_i^p-1\bigg|\ge
		\max\{t,t^p\}\bigg)
		\\
		& \qquad \le 2\exp
		\Big(-\frac{c'N}{R^2}\min\{\max\{t,t^p\}^2,\max\{t,t^p\}\}\Big).
	\end{align*}
	Thus, since $\bP(|Y|\ge t)\ge \bP(Y\ge t)$ for any random variable $Y$, we 
	obtain~\eqref{norm:concentration} from Fact III and conclude the proof.
\end{proof}

\begin{proof}[Proof of Fact I]
	From the $\psi_p$ norm definition,
	$\|X\|_{\psi_p}=\inf\big\{t>0\,|\,\bE\big[e^{(|X|/t)^p}\big]\le
	2\big\} = \inf\big\{t>0\,|\,t\ge |\mu|/(\ln 2)^\frac{1}{p}\big\}
	=|\mu|/(\ln 2)^\frac{1}{p}$,
	which establishes the result.
\end{proof}

\begin{proof}[Proof of Fact II]
	Assume first that $z<1$. Then, we have that
	$|z^p-1|=1-z^p>1-z\ge\delta\ge\delta^p$. Next, let $z\ge 1$. Then,
	we get $|z^p-1|=z^p-1\ge z-1\ge \delta$. In addition, when
	$\delta^p\ge \delta$, namely, when $\delta\ge 1$, we have that
	$z^p-(z-1)^p\ge 1$, and hence, $|z^p-1|=z^p-1\ge(z-1)^p\ge\delta^p$.
\end{proof}

\begin{proof}[Proof of Fact III]
	We consider two cases. Case (i): $0\le \delta\le 1\Rightarrow
	\delta\ge\delta^p \Rightarrow u=\max\{\delta,\delta^p\}=\delta
	$. Then $ \min\{u,u^2\}=\min\{\delta,\delta^2\}=\delta^2$. Case
	(ii): $\delta>1\Rightarrow\delta\le\delta^p\Rightarrow
	u=\max\{\delta,\delta^p\}=\delta^p$. Then
	$\min\{u,u^2\}=\min\{\delta^p,\delta^{2p}\}=\delta^p$. Thus, 
	we get  that $\min\{u,u^2\}=\alpha_p(\delta)$ for all $\delta\ge 0$.
\end{proof}

\begin{proof}[Proof of Theorem~\ref{thm:ambiguity:over:horizon}]
	Most of the proof is a verbatim repetition of the arguments employed
	for the proofs of Theorem~\ref{thm:ambiguity:set:noise} and the
	results invoked therein. Here we only provide the relevant
	modifications. For the nominal part of the ambiguity radius it
	suffices to establish that the support of the distribution of
	$\bm\xi_{\bm\ell}$ (in the infinity norm) is in the ball
	$B_{\infty}^{\widetilde\ell d}(\rho_{\bm\xi_{\bm\ell}})$ with
	$\rho_{\bm\xi_{\bm\ell}}$ given by \eqref{rhoxiell:stacked:states},
	which follows from \eqref{pushed:forward:ditributiuon} and the fact
	that the stack state vector $\bm\xi_{\bm\ell}$ over the horizon
	satisfies
	\begin{align*}
		\|\bm\xi_{\bm\ell}\|_{\infty}
		\le\max_{\ell\in[\ell_1:\ell_2]}\|\xi_\ell\|_{\infty}.
	\end{align*}
	
	For the noise radius, we get in analogy to the proof of 
	Lemma~\ref{lemma:empirical:vs:true} that 
	\begin{align*}
		W_p(\widehat P_{\bm\xi_{\bm\ell}}^N,P_{\bm\xi_{\bm\ell}}^N)
		\le\Big(\frac{1}{N}\sum_{i=1}^N\|\bm e_{\bm\ell}^i\|^p\Big)^\frac{1}{p}
	\end{align*}
	with $\bm e_{\bm\ell}^i:=(e_{\ell_1}^i,\ldots,e_{\ell_2}^i)$. Then we get as 
	in Lemma~\ref{lemma:empirical:vs:true} that   
	\begin{align*}
		\|\bm e_{\bm\ell}^i\|\le \sum_{\ell=\ell_1}^{\ell_2} \|e_{\ell}^i\|\le 
		\widetilde\fM(z^i,\bm\omega^i)+\widetilde\fE(\bm v^i),
	\end{align*}
	with
	\begin{align*}
		\widetilde\fM(z^i,\bm\omega^i):= \sum_{\ell=\ell_1}^{\ell_2} 
		\fM(z^i,\bm\omega^i;\ell)\qquad \widetilde\fE(\bm 
		v^i)\equiv\widetilde\fE^i:= \sum_{\ell=\ell_1}^{\ell_2}\fE^i(\ell)
	\end{align*}
	and $\fM(z^i,\bm\omega^i;\ell)\equiv\fM(z^i,\bm\omega^i)$, 
	$\fE^i(\ell)\equiv\fE^i$ as given in the proof of that lemma. Note that in 
	exact analogy to the proofs of Lemmas~\ref{lemma:empirical:vs:true} 
	and~\ref{lemma:Orlicz:bounds}, the 
	constants $\widetilde\fM_w$, $\widetilde\fM_v$, and $\widetilde\fC_v$ in the 
	statement of the theorem constitute upper bounds for each 
	$\widetilde\fM(z^i,\bm\omega^i)$ over $B_\infty^d(\rho_{\xi_0})\times
	B_\infty^{\ell q}(\rho_w)$, $\|\widetilde\fE^i\|_p$, and 
	$\|\widetilde\fE^i\|_{\psi_p}$, respectively, whereas $\widetilde\fm_v$ is a 
	positive lower bound for each $\|\widetilde\fE^i\|_{\psi_p}$. The remainder 
	of the proof follows the exact same steps as the proofs of 
	Proposition~\ref{prop:hatepsN:guarantees} and 
	Theorem~\ref{thm:ambiguity:set:noise}. 
\end{proof}

\begin{proof}[Proof of Lemma~\ref{lemma:Wasserstein:convolution}]
	The proof of the first part is found inside the proof of
	\cite[Lemma~1]{GCP-AP:16}. To prove the second part, let $\pi$ be an
	optimal coupling of $P_1$ and $P_2$ for the Wasserstein distance
	$W_p(P_1,P_2)$ and define
	\begin{align*}
		\widetilde\pi(B):=\int_{\Rat{d}\times\Rat{d}}\int_{\Rat{d}}
		\bone_B(x,y+z)Q(dz)\pi(dx,dy)
	\end{align*}
	for any measurable subset $B$ of $\Rat{d}\times\Rat{d}$. Then its 
	marginals satisfy
	\begin{align*}
		\widetilde\pi(\Rat{d}\times 
		B_y) & =\int_{\Rat{d}\times\Rat{d}}\int_{\Rat{d}}
		\bone_{B_y}(y+z)Q(dz)\pi(dx,dy) \\
		& = \int_{\Rat{d}}\int_{\Rat{d}}
		\bone_{B_y}(y+z)Q(dz)P_2(dy)=P_2\star Q(B_y)
	\end{align*}
	and $\widetilde\pi(B_x\times\Rat{d})=P_1(B_x)$, which is derived by 
	analogous arguments. In addition, 
	\begin{align*}
		&
		\bigg(\int_{\Rat{d}\times\Rat{d}}\|x-y\|^p\widetilde\pi(dx,dy)\bigg)^{\frac{1}{p}}
		=\bigg(\int_{\Rat{d}\times\Rat{d}}\int_{\Rat{d}}
		\|x-(y+z)\|^pQ(dz)\pi(dx,dy)\bigg)^{\frac{1}{p}} \\
		& \quad =\bigg(\int_{\Rat{d}\times\Rat{d}}\int_{\Rat{d}}
		\|(x-y)+z\|^pQ(dz)\pi(dx,dy)\bigg)^{\frac{1}{p}} \\
		& \quad \le \bigg(\int_{\Rat{d}\times\Rat{d}}\int_{\Rat{d}}
		\|x-y\|^pQ(dz)\pi(dx,dy)\bigg)^{\frac{1}{p}}
		+\bigg(\int_{\Rat{d}\times\Rat{d}}\int_{\Rat{d}}
		\|z\|^pQ(dz)\pi(dx,dy)\bigg)^{\frac{1}{p}} \\
		& \quad \le W_p(P_1,P_2)+\bigg(\int_{\Rat{d}\times\Rat{d}}
		q^p\pi(dx,dy)\bigg)^{\frac{1}{p}}= W_p(P_1,P_2)+q,
	\end{align*}
	where we used the definition of $\widetilde\pi$ in the first
	equality, the triangle inequality for the space
	$L_p(\Rat{d}\times\Rat{2d};Q\otimes\pi)$ in the first inequality,
	and that $\big(\int_{\Rat{d}}\|x\|^pQ(dx)\big)^{\frac{1}{p}}\le q$ in 
	the last inequality. This concludes also the second part of the proof.
\end{proof}

\subsection{Explicit constants in the concentration inequalities} 
\label{subsec:explicit:contants}

We first give an independent proof of the norm concentration
inequality in Proposition~\ref{prop:Bernstein:inequality}. This proof
entails the explicit derivation of the constant $c'=1/10$ therein,
which is the same as that in
Proposition~\ref{prop:concentration:around:mean}.  We note that a
general concentration result with the same decay rates as
Proposition~\ref{prop:concentration:around:mean} can also be found in
\cite[Exercise 2.27, Page 51]{SB-GL-PM:13}, however, without the
explicit characterization of the involved constants. We exploit an
equivalent characterization of sub-exponential random variables stated
next.  This characterization can be found in~\cite[Proposition~2.7.1
and Exercise~2.7.2]{RV:18}, but here we give the exact constants
 and the necessary modifications of the corresponding proofs.

\begin{lemma}\longthmtitle{Properties of sub-exponential random
    variables}\label{lemma:subexponential:properties}
  Let $X$ be sub-exponential. Then:
  \begin{enumerate}
  \item The tails of $X$ satisfy
    \begin{align*}
      \bP(|X|\ge t)\le 2\exp\big(-t/\|X\|_{\psi_1}\big)\quad \text{for all}\;t\ge 0. 
    \end{align*}   
  \item The moments of $X$ satisfy
    \begin{align*}
      \|X\|_{L^p}^p=\bE\big[|X|^p\big]\le 2p!\|X\|_{\psi_1}^p\quad \text{for all 
        integers}\;p\ge 1. 
    \end{align*}  
  %
  \item If additionally $\bE[X]=0$, the moment generating function of
    $X$ satisfies
    \begin{align*}
      \bE\big[\exp(\lambda X)\big]\le
      \exp\big(2\fa\|X\|_{\psi_1}^2\lambda^2\big),\quad \text{for
        all}\;\fa>1\;\text{and}\;\lambda\;\text{with}\;|\lambda|\le
      \frac{\fa-1}{\fa\|X\|_{\psi_1}}.
    \end{align*}   
  \end{enumerate}
\end{lemma}  
\begin{proof}
  To show \emph{(i)}, we use Markov's inequality and the definition of
  the $\psi_1$ norm. In particular, we have
  \begin{align*}
    \bP(|X|\ge t) & =\bP\big(\exp\big(|X|/\|X\|_{\psi_1}\big)\ge
    \exp\big(t/\|X\|_{\psi_1}\big)\big)
    \\
    & \le
    \bE\big[\exp\big(|X|/\|X\|_{\psi_1}\big)\big]\exp\big(-t/\|X\|_{\psi_1}\big)
    =2\exp\big(-t/\|X\|_{\psi_1}\big).
  \end{align*}
  To show \emph{(ii)}, note that
  \begin{align*}
    \bE\big[|X|^p\big] & =\int_0^\infty\bP(|X|^p\ge
    u)du=\int_0^\infty\bP(|X|\ge
    t)pt^{p-1}dt \\
    & \overset{{\rm (i)}}{\le}
    2\int_0^\infty\exp\big(-t/\|X\|_{\psi_1}\big)pt^{p-1}dt  \\
    & \overset{s=t/\|X\|_{\psi_1}}{=}2\|X\|_{\psi_1}^pp\int_0^\infty
    e^{-s}s^{p-1}ds=2\|X\|_{\psi_1}^p\Gamma(p+1).
  \end{align*} 
  Hence, since $\Gamma(p+1)=p!$ for integer values of $p$, we get 
  \begin{align*}
    \bE\big[|X|^p\big]\le 2p!\|X\|_{\psi_1}^p.
  \end{align*}
  %
  %
  
  To show \emph{(iii)}, note first that since $\bE[X]=0$,
  \begin{align*}
    \bE\big[\exp(\lambda X)\big]=\bE\Big[1+\lambda
    X+\sum_{p=2}^{\infty}\frac{(\lambda
      X)^p}{p!}\Big]=1+\sum_{p=2}^{\infty}\frac{\lambda^p\bE\big[X^p]}{p!}\le
    1+\sum_{p=2}^{\infty}\frac{\lambda^p\bE\big[|X|^p]}{p!}.
  \end{align*}
  Thus, we get from \emph{(ii)} that  
  \begin{align*}
    \bE\big[\exp(\lambda X)\big]\le 
    1+2\sum_{p=2}^{\infty}(\|X\|_{\psi_1}\lambda)^p 
    =1+2\frac{\|X\|_{\psi_1}^2\lambda^2}{1-\|X\|_{\psi_1}\lambda},
  \end{align*}
  for all $|\lambda|<1/\|X\|_{\psi_1}$. Further, when 
  $|\lambda|\le \frac{\fa-1}{\fa\|X\|_{\psi_1}}\iff 
  1-\lambda\|X\|_{\psi_1}\ge1/\fa$, we 
  have 
  \begin{align*}
    1+2\frac{\|X\|_{\psi_1}^2\lambda^2}{1-\|X\|_{\psi_1}\lambda}
    \le 1+2\fa\|X\|_{\psi_1}^2\lambda^2\le \exp(2\fa\|X\|_{\psi_1}^2\lambda^2).
  \end{align*} 
\end{proof}

We next use this result to specify the constant $c'$ in
Proposition~\ref{prop:Bernstein:inequality}, giving its proof along
the lines of Theorem~2.8.1--Corollary~2.8.3 in~\cite{RV:18}.

\begin{proof}[Proof of Proposition~\ref{prop:Bernstein:inequality}]
  We denote $S=\frac{1}{N}\sum_{i=1}^NX_i$ and consider the real
  parameter $\lambda$. Using independence of the $X_i$'s we get from
  Markov's inequality
  \begin{align*}
    \bP(S\ge t) & =\bP(\exp(\lambda S)\ge \exp(\lambda t))
    \\
    & \le \exp(-\lambda t)\bE[\exp(\lambda S)]=\exp(-\lambda
    t)\prod_{i=1}^N\bE\Big[\exp\Big(\lambda\frac{1}{N}X_i\Big)\Big].
  \end{align*}
  Hence, from Lemma~\ref{lemma:subexponential:properties}(ii) applied
  to the random variables $\frac{1}{N}X_i$, for any $\fa>1$ and
  \begin{align} \label{lambda:constraint}
    |\lambda|\le\frac{\fa-1}{\fa\frac{1}{N}\max_{i\in[1:N]}\|X_i\|_{\psi_1}} 
  \end{align}
  it holds $ \bE\Big[\exp\Big(\lambda\frac{1}{N}X_i\Big)\Big]\le\exp
  \Big(\fa\frac{\lambda^2}{N^2}\|X_i\|_{\psi_1}^2\Big) $, for each
  $i\in[1:N]$. Consequently,
  \begin{align*}
    \bP(S\ge t)\le \exp\Big(-\lambda 
    t+2\fa\frac{\lambda^2}{N^2}\|X_i\|_{\psi_1}^2\Big). 
  \end{align*}
  Minimizing with respect to $\lambda$ under the constraint 
  \eqref{lambda:constraint} we get the optimizer
  \begin{align*}
    \lambda_*=\min\bigg\{\frac{tN^2}{4\fa\sum_{i=1}^N\|X_i\|_{\psi_1}^2}, 
    \frac{(\fa-1)N}{\fa\max_{i\in[1:N]}\|X_i\|_{\psi_1}}\bigg\}. 
  \end{align*}   
  Combining this with the elementary inequality
  $\alpha\widehat\lambda^2 - \beta
  \widehat\lambda\le-\frac{\beta}{2}\widehat\lambda$, which holds for
  any $\alpha,\beta>0$ and
  $\widehat\lambda\in[0,\frac{\beta}{2\alpha}]$, we have
  \begin{align*}
    & \bP(S\ge t) \le
    \exp\bigg(-\min\bigg\{\frac{t^2N^2}{8\fa\sum_{i=1}^N\|X_i\|_{\psi_1}^2},
    \frac{t(\fa-1)N}{2\fa\max_{i\in[1:N]}\|X_i\|_{\psi_1}}\bigg\}\bigg) \\
    & \le \exp\bigg(-\min\bigg\{\frac{t^2}
    {8\fa\big(\max_{i\in[1:N]}\|X_i\|_{\psi_1}\big)^2},
    \frac{t(\fa-1)}{2\fa\max_{i\in[1:N]}\|X_i\|_{\psi_1}}\bigg\}N\bigg) \\
    & \le
    \exp\bigg(-\min\bigg\{\frac{1}{8\fa},\frac{\fa-1}{2\fa}\bigg\}\min\bigg\{\frac{t^2}
    {\big(\max_{i\in[1:N]}\|X_i\|_{\psi_1}\big)^2},
    \frac{t}{\max_{i\in[1:N]}\|X_i\|_{\psi_1}}\bigg\}N\bigg)
  \end{align*}
  for all $\fa>1$. Taking into account that 
  $\min\big\{\frac{1}{8\fa},\frac{\fa-1}{2\fa}\big\}$ is $\frac{1}{8\fa}$ for 
  $1<\fa<5/4$ and $\frac{\fa-1}{2\fa}$ for $\fa\ge 5/4$, we can select 
  $\fa=5/4$, which maximizes this term and obtain the optimal decay rate   
  \begin{align*}
    \bP(S\ge t)\le \exp\bigg(-\frac{1}{10}\min\bigg\{\frac{t^2}
    {\big(\max_{i\in[1:N]}\|X_i\|_{\psi_1}\big)^2},
    \frac{t}{\max_{i\in[1:N]}\|X_i\|_{\psi_1}}\bigg\}N\bigg).
  \end{align*}
  Repeating the above arguments for the random variables
  $-\frac{1}{N}X_i$, we derive the same bound for $\bP(-S\ge t)$ and
  establish the result with $c'=1/10$.
\end{proof}

We next provide explicit constants $C$ and $c$ for the nominal
ambiguity radius $\eps_N$ given by~\eqref{epsN:dfn} when $p
<d/2$. Note that any other case can also be reduced to this at the
cost of increased conservativeness by embedding the distribution in a
higher-dimensional space.  Further, the most typical values of $p$ are
$p=1$, where general DRO problems admit the tractable reformulations
provided in~\cite{PME-DK:17}, and $p=2$, where the dual optimization
problem admits certain convenient quadratic terms, which for instance
facilitate taking gradients~\cite{AC-JC:20-tac}. Thus, one can use the
precise concentration results for reasonably low-dimensional data.  To
obtain the desired constants, we exploit results
from~\cite{EB-TLG:14,SD-MS-RS:13}. In particular,
from~\cite[Proposition A.2]{EB-TLG:14}, we have the following
concentration inequality which quantifies how the Wasserstein distance
between the true and the empirical distribution concentrates around
its expected value.

\begin{proposition}\longthmtitle{Concentration around empirical
    Wasserstein mean}\label{prop:concentration:result}
  Assume that the probability measure $\mu$ is supported on the
  compact subset $B$ of $\Rat{d}$ (with the Euclidean norm). Then,
  \begin{align} \label{concentration:ineq} \bP(W_p(\mu^N,\mu)\ge
    \bE[W_p(\mu^N,\mu)]+t)\le
    e^{-Nt^{2p}/(2\widetilde\rho^{2p})} \qquad \forall t\ge 0,
  \end{align}
  where $\widetilde\rho=\diameter{2}{B}$.
\end{proposition}


We will also use~\cite[Proposition~1 and Remark~4]{SD-MS-RS:13}, which
give the following bound for the expected Wasserstein distance between
the empirical and actual distribution.

\begin{proposition}\longthmtitle{Decay of empirical Wasserstein
    mean}\label{prop:W:concentration:around:mean} 
  For any probability measure $\mu$ supported on $[0,1)^d$ and
  $p<d/2$, $ \bE\big[W_p(\mu^N,\mu)\big] \le C_\star N^{-1/d}$,
  with
  \begin{align}
    \label{W:concentration:around:mean}
    C_\star:=\sqrt{d}2^{(d-2)/(2p)}\left(\frac{1}{1-2^{p-d/2}}
      +\frac{1}{1-2^{-p}}\right)^{1/p}.
    \end{align}
\end{proposition}

Combining Propositions~\ref{prop:concentration:around:mean} 
and~\ref{prop:W:concentration:around:mean} we get the following explicit 
characterization of the nominal ambiguity radius.  

\begin{proposition}\longthmtitle{Explicit concentration inequality
    constants}\label{prop:explicit:ambiguity:radius} 
  Assume that the probability measure $\mu$ is supported on
  $B\subset\Rat{d}$ with
  $\rho:=\frac{1}{2}\diameter{\infty}{B}<\infty$ and that
  $p<d/2$. Then, we can select the nominal ambiguity radius
  \begin{align*}
    \eps_N(\beta,\rho):=2\rho \big(C_\star N^{-\frac{1}{d}}
    +\sqrt{d}(2\ln\beta^{-1})^{\frac{1}{2p}}N^{-\frac{1}{2p}}\big).
  \end{align*}
\end{proposition}
\begin{proof}
  Since the Wasserstein distance of the dilation of two distributions
  in a vector space by a factor is equal to this factor times their
  original Wasserstein distance (as exploited e.g., in
  \cite[Proposition 3.2]{DB-JC-SM:21-tac}), we have from
  Proposition~\ref{prop:W:concentration:around:mean} that
  $\bE(W_p(\mu^N,\mu)) \le 2\rho C_\star N^{-1/d}$. Substituting
  the latter in \eqref{concentration:ineq} and taking into account
  that $\diameter{2}{B}\le\sqrt{d}\diameter{\infty}{B}$, i.e., that
  $\widetilde\rho=\sqrt{d} 2\rho$, we get
  \begin{align*} 
    \bP(W_p(\mu^N,\mu)\ge 2\rho C_\star N^{-\frac{1}{d}}+t)\le 
    e^{-Nt^{2p}/(2d^p{(2\rho)}^{2p})} \qquad \forall t\ge 0.
  \end{align*}
  Set $\eps:={2\rho}C_\star N^{-\frac{1}{d}}+t$ and
  $\beta:=e^{-Nt^{2p}/(2d^p{(2\rho)}^{2p})}
  \iff
  t=\sqrt{d}{2\rho}(2\ln\beta^{-1})^{\frac{1}{2p}}N^{-\frac{1}{2p}}$. Then
  $
  \bP(W_p(\mu^N,\mu)\le\eps)
  \ge 1-\beta $ for all $\beta\in(0,1)$ and
  \begin{align*}
    \eps\equiv\eps_N(\beta,\rho)={2\rho}\big(C_\star N^{-\frac{1}{d}} 
    +\sqrt{d}(2\ln\beta^{-1})^{\frac{1}{2p}}N^{-\frac{1}{2p}}\big). 
  \end{align*}
\end{proof}

We also give an explicit ambiguity radius expression in terms of a
single exponential inequality as in \eqref{epsN:dfn} in the following
result.

\begin{corollary}\longthmtitle{Alternative explicit constants} 
  Under the assumptions of
  Proposition~\ref{prop:explicit:ambiguity:radius} we can select the
  nominal ambiguity radius
  \begin{align*}
    \eps_N(\beta,\rho):={2\rho}\bigg(\frac{\ln\big(C^\star\beta^{-1}\big)} 
    {c^\star}\bigg)^{\frac{1}{d}}N^{-\frac{1}{d}},
  \end{align*}
  with $C^\star:=\frac{C_\star^d}{2\sqrt{d}^d}$ and 
  $c^\star:=\frac{1}{2^d\sqrt{d}^d}$
\end{corollary}
\begin{proof}
  Note first that $e^{-Nt^{2p}/(2d^p{(2\rho)}^{2p})}\le
  e^{-Nt^d}/(2\sqrt{d}^d{(2\rho)}^d)$ when $t\in[0,\sqrt{d}{2\rho}]$ (for
  $t>\sqrt{d}{2\rho}$ the probability of interest is zero). Thus,
  using the inequality $a^{\frac{1}{q}}+b^{\frac{1}{q}}\le
  \big(2^{q-1}(a+b)\big)^{\frac{1}{q}}$ for $q\ge 1$, we get in
  analogy to the proof of
  Proposition~\ref{prop:explicit:ambiguity:radius} that
  \begin{align*}
    \eps 
    & ={2\rho}\big(CN^{-\frac{1}{d}}
    +\sqrt{d}(2\ln\beta^{-1})^{\frac{1}{d}}N^{-\frac{1}{d}}\big)    
    ={2\rho}\big(C_\star+\sqrt{d}(2\ln\beta^{-1})^{\frac{1}{d}}\big)N^{-\frac{1}{d}}
    \\
    & ={2\rho}\Big(\big(C_\star^d\big)^\frac{1}{d}
    +\big(2\sqrt{d}^d\ln\beta^{-1}\big)^{\frac{1}{d}}\Big)N^{-\frac{1}{d}}
    \le{2\rho}\Big(2^{d-1}\big(C_\star^d
    +2\sqrt{d}^d\ln\beta^{-1}\big)\Big)^{\frac{1}{d}}N^{-\frac{1}{d}}
    \\
    & ={2\rho}\big(2^{d-1}C_\star^d
    +2^d\sqrt{d}^d\ln\beta^{-1}\big)^{\frac{1}{d}}N^{-\frac{1}{d}}
    ={2\rho}\bigg(\frac{\frac{C_\star^d}{2\sqrt{d}^d}
      +\ln\beta^{-1}}{\frac{1}{2^d\sqrt{d}^d}}\bigg)^{\frac{1}{d}}N^{-\frac{1}{d}}
    \\
    &
    ={2\rho}\bigg(\frac{\ln\big(\frac{C_\star^d}{2\sqrt{d}^d}\beta^{-1}\big)}
    {\frac{1}{2^d\sqrt{d}^d}}\bigg)^{\frac{1}{d}}N^{-\frac{1}{d}}
    \equiv{2\rho}\bigg(\frac{\ln\big(C^\star\beta^{-1}\big)}
    {c^\star}\bigg)^{\frac{1}{d}}N^{-\frac{1}{d}},
  \end{align*}
  with $C^\star$ and $c^\star$ as given in the statement.
\end{proof}

\subsection{Sub-Gaussian norms of Gaussian mixtures}

Here we discuss how to compute the $\psi_2$ norm, i.e., the
sub-Gaussian norm of a random variable with a Gaussian mixture
distribution.

\noindent $\triangleright$ \textit{Fact IV.} For $X\sim\mathcal
N(0,1)$, it holds that $\|X\|_{\psi_2}=\sqrt{8/3}$. \quad
$\triangleleft$

\begin{proof}
  By definition, $\|X\|_{\psi_2}=\inf\{t>0\,|\,\bE[\exp(X^2/t^2)] \le
  2\}$. Therefore, we seek to determine
  $\inf\big\{t>0\,|\,\frac{1}{\sqrt{2\pi}}\int_\Rat{}
  \exp\big(-x^2\big(\frac{1}{2}-\frac{1}{t^2}\big)\big)dx\le 2\big\}$.
  Setting $\frac{1}{2\sigma^2}=\frac{1}{2}-\frac{1}{t^2}$, namely,
  $\sigma\equiv\sigma(t)=\sqrt{\frac{t^2}{t^2-2}}$, the expression
  becomes
  \begin{align*}
    \inf\bigg\{t>0\,\Big|\,\frac{\sigma}{\sqrt{2\pi}\sigma}\int_\Rat{} 
    \exp\bigg(-\frac{x^2}{2\sigma^2}\bigg)dx\le 2\bigg\}=
    \inf\bigg\{t>0\,\Big|\,\sqrt{\frac{t^2}{t^2-2}}\le 2\bigg\}=\sqrt{8/3}. 
  \end{align*} 
\end{proof}

\noindent $\triangleright$ \textit{Fact V.} For $X\sim\mathcal
N(\mu,\sigma^2)$, it holds that
$\|X\|_{\psi_2}=\sigma\sqrt{8/3}+\mu/\sqrt{\ln 2}$. \quad
$\triangleleft$

\begin{proof}
  Note that $X=Y+\sigma Z$, with $Y=\delta_\mu$ and $Z=\mathcal
  N(0,1)$. Since $\|\cdot\|_{\psi_2}$ is a norm, we get from Fact I in
  the proof of Proposition~\ref{prop:Bernstein:inequality} and Fact IV
  above that $\|X\|_{\psi_2}\le
  \|Y\|_{\psi_2}+\sigma\|Z\|_{\psi_2}=\mu/\sqrt{\ln
    2}+\sigma\sqrt{8/3}$.
\end{proof}

\noindent $\triangleright$ \textit{Fact VI.} Given arbitrary
distributions $\nu_i$, let $X_i~\sim \nu_i$, $i=1,\ldots,n$ and
$X\sim\sum_{i=1}^n c_i\nu_i$, with $\sum_{i=1}^nc_i=1$, $c_i
\ge 0$. Then $\|X\|_{\psi_2}\le
\max_{i=1,\ldots,n}\|X_i\|_{\psi_2}$.\quad $\triangleleft$
 
\begin{proof}
  From the definition of the ${\psi_2}$ norm,
  \begin{align*}
    \|X\|_{\psi_2} &
    =\inf\bigg\{t>0\,\Big|\,\sum_{i=1}^nc_i\int_\Rat{}\exp\big(x^2/t^2\big)
    {\nu_i(dx)}\le 2\sum_{i=1}^nc_i\bigg\}
    \\
    & \le \inf\bigg\{t>0\,\Big|\,\int_\Rat{}\exp\big(x^2/t^2\big)
    {\nu_i(dx)}\le 2\;\forall i=1,\ldots,n\bigg\}
    \\
    &
    =\max_{i=1,\ldots,n}\inf\bigg\{t>0\,\Big|\,\int_\Rat{}\exp\big(x^2/t^2\big)
    {\nu_i(dx)}\le 2\bigg\}=\max_{i=1,\ldots,n}\|X_i\|_{\psi_2}.
  \end{align*} 
\end{proof}

The following result is a consequence of Facts V and VI.

\begin{proposition}\longthmtitle{Sub-Gaussian norm of Gaussian
    mixture}\label{prop:Gaussian:mixture:psi2:norm} 
  Let $X\sim\sum_{i=1}^nc_i\mathcal N(\mu_i,\sigma_i^2)$, with
  $\sum_{i=1}^nc_i=1$,  $c_i \ge 0$. Then,  
  $\|X\|_{\psi_2}\le
  \max_{i=1,\ldots,n}\{\sigma_i\sqrt{8/3}+\mu_i/\sqrt{\ln 2}\}$.
\end{proposition}

\bibliographystyle{siamplain}
\bibliography{alias,SM,JC,SMD-add} 

\end{document}